\documentclass[11pt]{amsart}

\usepackage{amssymb}
\usepackage{amsthm}
\usepackage{amsbsy}
\usepackage{subfigure}
\usepackage{amsmath,mathrsfs}
\usepackage{graphicx}

\usepackage{comment}
\usepackage[usenames,dvipsnames]{xcolor}
\usepackage[margin = 1.1in]{geometry}
\usepackage{enumerate}
\usepackage[toc,page]{appendix}
\usepackage{lineno}
\usepackage{color}
\usepackage{hyperref}
\usepackage[ruled,vlined,linesnumbered]{algorithm2e}

\usepackage{ulem}

\SetKwInOut{Parameter}{\sl Algorithm Inputs}

\newcommand{\blue}{}
\definecolor{mygreen}{rgb}{0.1,0.75,0.2}

 \newtheorem{thm}{Theorem}[section]
 
 \newtheorem{lem}[thm]{Lemma}

 \newtheorem{rem}[thm]{Remark}

\numberwithin{equation}{section}

\DeclareMathOperator{\st}{s.t.~}

\newcommand{\la}{\langle}
\newcommand{\ra}{\rangle}
\newcommand{\pt}{\partial}
\newcommand{\eps}{\varepsilon}
\newcommand{\Ps}{\mathscr{P}}
\newcommand{\Pc}{\mathcal{P}}

\providecommand{\bbs}[1]{\left(#1\right)}
\newcommand{\aaa}[1]{\begin{equation}
begin{aligned} #1 \end{aligned}
\end{equation}}

\newcommand{\ud}{\,\mathrm{d}}
\newcommand{\8}{\infty}
\newcommand{\mm}{\mathcal{M}}
\newcommand{\nn}{\mathcal{N}}
\newcommand{\mx}{\mathbf{x}}
\newcommand{\my}{\mathbf{y}}
\newcommand{\mv}{\mathbf{v}}
\newcommand{\mz}{\mathbf{z}}

\newcommand{\piw}{\frac{(\pi_i+\pi_j)|\Gamma_{ij}|}{2}}

\newcommand{\hs}{\mathcal{H}}

\newcommand{\bR}{\mathbb{R}}
\newcommand{\bE}{\mathbb{E}}
\newcommand{\bP}{\mathbb{P}}

\newcommand{\xb}{{\mathbf x}}

\DeclareMathOperator*{\argmax}{argmax}
\DeclareMathOperator*{\argmin}{argmin}
\renewcommand{\vec}[1]{\ensuremath{\boldsymbol{#1}}}
\newcommand{\abs}[1]{\left| #1 \right|}
\newcommand{\norm}[1]{\left\|{#1}\right\|}

\begin{document}

\title[Transition path theory for Langevin dynamics on manifolds ]{Transition path theory for Langevin dynamics on manifolds: optimal control and data-driven solver}

\author[Y. Gao]{Yuan Gao}
\address{Department of Mathematics, Purdue University, West Lafayette, IN, USA}
\email{gao662@purdue.edu}

\author[T. Li]{Tiejun Li}
\address{LMAM and School of  Mathematical Sciences, Peking University, Beijing 100871, China}
\email{tieli@pku.edu.cn}

\author[X. Li]{Xiaoguang Li}
\address{MOE-LCSM, School of Mathematics and Statistics, Hunan Normal University, Changsha 410081,  China}
\email{lixiaoguang@hunnu.edu.cn}

\author[J.-G. Liu]{Jian-Guo Liu}
\address{Department of Mathematics and Department of
  Physics, Duke University, Durham, NC, USA}
\email{jliu@math.duke.edu}

\begin{abstract}
We present a data-driven point of view for   rare events, which represent conformational transitions in biochemical reactions modeled by over-damped Langevin dynamics  on manifolds in high dimensions.   {\blue We first reinterpret the transition state theory and the transition path theory  from the optimal control viewpoint. Given  point clouds sampled from a    reaction dynamic, we construct a discrete Markov process based on an approximated Voronoi tesselation. We use the constructed Markov process  to compute a discrete committor function whose level set automatically orders the point clouds. Then based on the committor function, an  optimally controlled random walk on point clouds is constructed and utilized to efficiently sample  transition paths, which become an  almost sure event in $O(1)$ time instead of a rare event in the original reaction dynamics.  To compute the mean transition path efficiently,  a local averaging algorithm based on the optimally controlled random walk   is developed, which adapts the finite temperature string method to the controlled Monte Carlo samples.  Numerical examples on sphere/torus including a conformational transition for the alanine dipeptide in vacuum  are conducted to illustrate the data-driven solver for the  transition path theory on point clouds.
The  mean transition path obtained via the controlled Monte Carlo simulations highly coincides with the computed dominant transition path in the transition path theory.}
\end{abstract}


\maketitle
\section{Introduction}

Complex molecular dynamics in chemical/biochemical reactions usually have cascades of timescales. For instance, the molecular vibrations occur in femtosecond  time scale, while the conformational transitions occur in microsecond time scale. Assume the states $\{\xb(t)\}$ of the original molecular dynamics are  in a high dimensional space $\mathbb{R}^p$, and suppose the most important slow  dynamics such as conformational transitions can be described by a reduced over-damped Langevin dynamics in terms of reaction coordinates $\my$ on an intrinsic low-dimensional manifold $\nn\subset \mathbb{R}^\ell$, where $\ell\ll p$ \cite{coifman2006diffusion}. Then the slow dynamics on $\nn$ is guided by a reduced free energy $U_\nn(\my)$, $\my \in\nn$, whose local minimums indicate several metastable states (say $a,b$) of the dynamics.   Those conformational transitions from one metastable state $a$ to another metastable state $b$  are rare (but significant) events compared with typical relaxation dynamics in each energy basin. Thus  efficient simulation and computation of transition rates or reaction pathways for the conformational transitions are challenging and important problems, which has been one of the core subjects in applied mathematics in recent years; see recent  review in  \cite{E_Vanden-Eijnden_2010}.

To make the discussion precise, we denote the over-damped Langevin dynamics of $\my_t$ by
\begin{equation}\label{sde-y}
\ud {\my_t} = - \nabla_\nn   U_{\nn}(\my_t) \ud t + \sqrt{2\eps} \sum_{i=1}^d \tau^\nn_{i}(\my_t)\otimes \tau^\nn_i(\my_t) \circ \ud B_t,
\end{equation}
where  $\eps>0$ corresponds to the thermal energy in physics, the symbol $\circ$ means the Stratonovich  integration, $B_t$ is $\ell$-dimensional Brownian motion, $d$ is the intrinsic dimension of the manifold $\nn$ and $\{\tau^\nn_i;\, 1\leq i \leq d\}$ are orthonormal basis of tangent plane $T_{\my_t}\nn$.
Here $\nabla_\nn := \sum_{i=1}^d \tau^{\nn}_i \nabla_{\tau^{\nn}_i}= \sum_{i=1}^d \tau^{\nn}_i \otimes \tau^{\nn}_i \nabla$ is the surface gradient and $\nabla_{\tau^{\nn}_i}=\tau^{\nn}_i \cdot \nabla$ is the tangential derivative in the direction of $\tau^{\nn}_i$. Given the manifold $\nn$ and potential $U_{\nn}(\my)$, when $\eps$ is small, the study of rare events by direct simulation of \eqref{sde-y} is not feasible. This motivates the need of theoretical developments. In the limit $\eps\rightarrow 0$, the optimal transition path problem can be well described through the large deviation theory \cite{Freidlin_Wentzell_2012}. This was formulated as the minimum action method (MAM) \cite{MAM}, and was extended to the manifold case in \cite{zhou2016}. In the gradient case, the optimal transition path by MAM is actually the minimal energy path (MEP) connecting two metastable states, which was realized by the string method \cite{string}. The string method was further extended to the finite $\eps$ case, i.e. the finite temperature string method for gradient systems \cite{FTSM}. In the general finite noise case, the transition path theory (TPT) was first proposed by \textsc{E and Vanden-Eijnden} in \cite{weinan2006} to obtain  the transition paths and transition rates, etc., by the committor function $q$ (see \eqref{dn}).  A mathematically rigorous interpretation of TPT was given in \cite{Lu_Nolen_2015} (see precise statement in \cite[Theorem 1.7]{Lu_Nolen_2015}). The generalization of TPT to Markov jump process was given in \cite{MMS2009}.

We are concerning the rare event study from a data-driven point of view in this paper. In many cases, the manifold $\nn$ is not explicitly known and we assume only the  point clouds $\{\mathbf x_i\}_{i=1:n}$ are available from some physical dynamics on an unknown {\blue $d$ dimensional closed smooth} Riemannian submanifold $\mm \subset \mathbb{R}^p$.  {\blue We assume} one can  learn the reaction coordinates $\my=\Phi(\mathbf x): \mm\hookrightarrow \mathbb{R}^\ell$  by using  $\{\mathbf x_i\}_{i=1:n}$ in $\mathbb{R}^p$. Thus $\nn= \Phi(\mm)$ is a {\blue $d$ dimensional closed smooth} submanifold of  $\bR^\ell$ and one can represent the high dimensional data $\{\mathbf x_i\}\subset\mm\subset\bR^p$ as $\{\my_i\}=\{\Phi(\mx_i)\}\subset \nn\subset \bR^\ell$ in the low dimensional space. {\blue In general, this dimension reduction step is very challenging. Besides the standard diffusion map nonlinear dimension reduction \cite{coifman2006diffusion}, some reinforced learning methods based on build-in domain knowledge are developed recently; c.f. \cite{ZWE}.  As for a real example, we simulate a simple, manageable alanine dipeptide with 22 atoms to collect a full atomic molecular dynamics result for $\mx\in \bR^{66}$. Its lower energy states are  mainly described by two backbone dihedral angles $\phi\in [-\pi, \pi)$ and $\psi\in [-\pi, \pi)$, so the reaction coordinates $\my$ is chosen to be in a torus $\mathcal{N} \subset \bR^{3}$; see Fig. \ref{fig:aa} and Example 3 below.}
With the learned reaction coordinates $\my$, the main goal is to effectively simulate the conformational transitions on $\nn$.

Our main contributions are in two folds: (i) we give the {\it stochastic optimal control reinterpretations of the committor function} in TPT;  (ii)  {\blue   adapting similar idea as the  finite temperature string method,   we {\it proposed a data-driven solver for finding a mean transition path on manifold, which are constructed using the level-set defined by committor function and   taking advantage of an optimally controlled Monte Carlo simulation.}} They are detailed as below.

\vspace*{2mm}
\paragraph{\bf Interpretation from optimal control viewpoint}
The study of rare events from the optimal control viewpoint was pioneered in \cite{HS_2012, Hartmann_Banisch_Sarich_Badowski_S_2013} and further investigated in \cite{Hartmann_S_Zhang_2016, Hartmann_Richter_S_Zhang_2017}. Our goal in this part is  to rigorously reformulate
the effective conditioned process - constructed from committor function $q$ by \cite{Lu_Nolen_2015} - as
  a stochastic optimal controlled process,  in which the optimal feedback controls the original Markov process from one stable absorbing set $A$ to another absorbing set $B$ of the energy landscape $U_\nn$ with minimum cost. Before doing this,  we first reformulate the MEP finding problem from metastable states $a\in A$ to  $b\in B$ as a deterministic  optimal control problem by considering a controlled ODE with minimum cost \eqref{fw_c} in the infinite time horizon. {\blue This MEP is exactly the most probable path in the large deviation principle as noise $\eps \to 0$ in the Freidlin-Wentzell theory. Compared with the infinite optimal terminal time $T=+\8$, the corresponding stochastic optimal control problem at a fixed noise level $\eps>0$ is indeed easier in the sense that  the stopping time (the stochastic terminal time) is almost surely finite $\tau<+\8$. Thus the stochastic optimal control reinterpretation enables feasible computation of transition paths for practical scientific problems. } Precisely, for the system described by a Markov process  $\my_t$ on manifold $\nn$, e.g. \eqref{sde-y}, it will induce a measure on the path space $C([0,+\8); \nn)$, which can be regarded as a prior measure $P$. Then we need to construct a controlled Markov process $\tilde{\my}_t$ (see \eqref{oc_gamma}), which  induces a new measure $\tilde{P}$ on the path space. The additional control drives the trajectory $\tilde{\my}_t$ from $A$ to $B$ almost surely, while the transitions for the original process are  rare events.  From the stochastic optimal control viewpoint, we need to find a control function $\mv(\tilde{\my}_t)$ to realize the optimal change of measure from $P$ to $\tilde{P}$ such that the running cost (kinetic energy) and terminal cost (boundary cost hitting the absorbing set) are minimized. In Section \ref{sec_soc_th} and Theorem \ref{th_soc}, we will prove the forward committor function $q$ gives such an optimal control $\mv^*=2\eps \nabla \ln q$ that realizes the optimal change of measure on path space and thus realizes the transitions almost surely. These results also provide the basis for computing the  transition path through the Monte Carlo simulations for the controlled random walk and the local average mean path algorithm in the second part.

\vspace*{2mm}
\paragraph{\bf Data-driven solver for mean transition path on manifolds}
Our goal in this part is to take the advantage of the optimal control reinterpretation  above in its discrete analogies to design a data-driven solver, which efficiently finds  a mean transition path on a manifold suggested by point clouds. Given the point clouds $\{\my_i\}_{i=1:n}$, we construct a discrete Markov process  on $\{\my_i\}\subset \nn$ based on an approximated Voronoi tesselation for $\nn$, which incorporates both the equilibrium and manifold information.  The assigned transition probability between the nearest neighbor points (adjacent points identified by Voronoi tesselation) enables us to efficiently compute the discrete committor function $\{q_i\}_{i=1:n}$ and related quantities.  {\blue In Section \ref{sec_soc_D}, based on the constructed discrete Markov process, we derive an optimally controlled Markov process on point clouds with the associated controlled generator $Q^q$. More importantly,  the corresponding effective equilibrium under control is simply given by $\pi^e\propto q^2 \pi$ which preserves the detailed balance property. First, this  enables an efficient controlled Monte Carlo simulations for the new almost sure event in $O(1)$ time instead of the rare event in the original process. Moreover, 
adapting the idea from the finite temperature string method \cite{FTSM, Ren2005Transition}, we use  a Picard iteration to find the mean  transition path based on the  controlled Monte Carlo samplers. 
The numerical construction of the optimally controlled random walk   is given in Section \ref{sec_soc_D} while the  TPT analysis and the local average mean path algorithms based on the controlled process sampling  are given in Section \ref{sec_com}. We  apply this methodology to simulate the rare transitions on sphere/torus with Mueller potential and the transition between different isomers of an alanine dipeptide with MD simulation data. The developed mean transition path algorithm based on the controlled process performs very well for both synthetic and real world examples,  which highly coincide with the computed dominant transition path in TPT.
}

The rest of this paper is organized as follows. In Sections ~\ref{sec2}-\ref{sec3}, we will make the connections between the optimal control theory and the transition state theory and the transition path theory, respectively. In Section~\ref{sec4}, we present the constructions of the discrete original/controlled Markov process  on point clouds. In Section~\ref{sec_com}, we present the detailed algorithms for finding mean transition paths, while in Section~\ref{sec_num}   numerical examples are conducted to show the validity of the proposed algorithms.

\section{ Optimal control viewpoint for the transition state theory }\label{sec2}
In this section, we will first consider the typical energy landscape $U_\nn$ which indicates two metastable states and guides the Langevin dynamics on $\nn$ (see Section \ref{sec_energy}). Then we reformulate and solve the minimal energy path in the transition state theory from  the deterministic optimal control viewpoint in Section \ref{sec2.2}. In Section \ref{sec_soc}, we briefly review the basic concepts for a stochastic optimal control problem in the infinite time horizon.

\subsection{Energy landscape in the transition state theory}\label{sec_energy}
A chemical reaction from  reactants $a\in\nn$ through a transition state $c\in\nn$ to   products $b\in\nn$ can be described by the reaction coordinate $\my$ and  a path on the reaction coordinate $\my(t)\in\nn$ with a pseudo-time $t\in[0,T]$ and  $\my(0)=a$, $\my(T)=b$.
This chemical reaction can be characterized by an underlying potential $U_{\nn}(\my)$ in terms of the reaction coordinate $\my\in\nn$, which is {\blue assumed to be smooth enough} and has a few deep wells separated by high barriers. Assume $a$ and $b$ are two local minimums (attractors) with the basins of attractors {\blue $A,\,B\subset\nn, A\cap {B}=\emptyset$}   and $\max(U_\nn(a), U_\nn(b))< \min U(\pt A \cup \pt B) $.   {\blue The associated  minimal energy barrier such that $\my(0)=a$, $\my(T)=b$  is given by
\begin{equation}
\min_{\my(\cdot)} \bbs{ \max_{t\in[0,T]} U_\nn(\my(t)) }.
\end{equation}
Then  pick a path $\my^*(\cdot) \in \argmin_{\my(\cdot)}  \bbs{ \max_{t\in[0,T]} U_\nn(\my(t)) },$
and define
\begin{equation}\label{saddle}
c=\argmax_{\my^*(t), t\in[0,T]}U_\nn(\my^*(t)).
\end{equation}
This state $c$ achieves the minimal energy barrier and we assume this type of $c$ is unique, called transition state.
Moreover, assume $U_\nn$ is a Morse function and the Morse index at saddle point (transition state $c$) is $1$, i.e., there is only one negative eigenvalue   for the Hessian matrix of $U_\nn$ in the neighborhood of $c$.
 The energy barrier to achieve the chemical reaction from $a$ to $b$   is $U_\nn(c)-U_\nn(a)$. With this assumption, the minimal energy path can be uniquely found from the following least action problem \cite{Freidlin_Wentzell_2012}.}

By \cite{Freidlin_Wentzell_2012}, the minimal energy path is  given by the least action
\begin{equation}\label{fw_c}
\begin{aligned}
\inf_{T>0}~ \inf_{\my\in C[0,T]; \my(0)=a, \my(T)=b} \int_0^T \frac12 |\dot{\my}+\nabla U_\nn(\my)|^2	 \ud t.
\end{aligned}
\end{equation}
For notation simplicity,  from now on we use $U$ as the shorthand notation of $U_\nn$, $\nabla$ as shorthand notation of $\nabla_\nn$ and $\nabla \cdot$ as shorthand notation of   the surface divergence defined as $\nabla_\nn \cdot  \xi = \sum_{i=1}^d \tau_i^\nn \cdot \nabla_{\tau_i^\nn} \xi.$
In the case that $U$ is a double-well potential with $a,b$ being two local minimums, the minimal energy path (MEP) is given by the combination of the solutions $\my(t), t\geq 0$ to \cite{Freidlin_Wentzell_2012}
\begin{equation}\label{mep}
\begin{aligned}
\dot{\my} =  \nabla U(\my), \quad \my(-\8)=a, \, \my(+\8)=c,\\
\dot{\my}=-\nabla U(\my), \quad \my(-\8)=c, \, \my(+\8)=b.
\end{aligned}
\end{equation}

\subsection{Minimal energy path as a deterministic optimal control problem}\label{sec2.2}
One can recast the least action problem \eqref{fw_c} as a deterministic optimal control problem,  which  determines the optimal trajectory  $\my$, the feedback control $\mv$, and the  optimal terminal
time $T$. Since the original dynamics is translation invariant w.r.t.  time (stationary),
 we introduce the control variable $\mv=\mv(\my(t))$. The least action problem in \eqref{fw_c} is equivalent to the following optimal control problem with running cost $L=\frac12 |\mv|^2$
\begin{equation} \label{oc}
\begin{aligned}
 \gamma=&\inf_{T, \mv} \int_{T_1}^T \frac12 |\mv(\my(t))|^2 \ud t\\
& \st \dot{\my} = \mv(\my) - \nabla U(\my), \,\, t\in(T_1,T), \quad  \my(T_1)=a,\,\, \my(T)=c.
\end{aligned}
\end{equation}
Here $\mv(\my)$ belongs to the $C^1$ tangent vector field on $\nn$.
Define the augmented Lagrangian function as
\begin{equation}
\mathcal{L}(\my,\dot{\my},\mv,\lambda):=\frac12 |\mv|^2- \lambda \cdot( \dot{\my} - \mv  + \nabla U(\my)),
\end{equation}
where $\lambda $ is the Lagrange multiplier. Then the corresponding Hamiltonian is
\begin{equation}
\mathcal{H}(\my, \mv, \lambda):= \mathcal{L}-\dot{\my} \frac{\pt \mathcal{L}}{\pt \dot{\my}} =  \frac12 |\mv|^2 + \lambda \cdot\bbs{\mv - \nabla U(\my)}.
\end{equation}
Then calculating the first variation of $\int_0^T \mathcal{L} \ud t$ with respect to perturbations $\tilde{\my}, \tilde{\mv}, \tilde{\lambda}, \tilde{T}$ gives the  Euler-Lagrange equation  to \eqref{oc}
\begin{equation}
\frac{\pt \mathcal{L}}{\pt \my}-\frac{\ud}{\ud t} \bbs{\frac{\pt \mathcal{L}}{\pt \dot{\my}}}=0,\quad \frac{\pt \mathcal{L} }{\pt \mv} = 0, \quad \frac{\pt \mathcal{L}}{\pt \lambda}=0, \quad \mathcal{H}(T_*)=0,
\end{equation}
which can be further simplified as
\begin{equation}\label{el}
\begin{aligned}
&\dot{\mv} = \nabla^2 U(\my) \mv, \quad
\dot{\my} = \mv - \nabla U (\my),\quad \my(T_1) = a, \quad \my(T)=c,\\
&-\frac{|\mv|^2}{2} + \nabla U(\my) \cdot \mv \Big|_{t=T_*}=0.
\end{aligned}
\end{equation}
By the Pontryagin maximum principle, the Euler-Lagrange equation can be expressed in terms of the corresponding Hamiltonian $\mathcal{H}$
\begin{equation}
\frac{\ud}{\ud t } \my = \frac{\pt \mathcal{H}}{\pt \lambda}, \quad \frac{\ud}{\ud t} \lambda = - \frac{\pt \mathcal{H}}{\pt \my}, \quad \frac{\pt \mathcal{H}}{\pt \mv}=0, \quad \mathcal{H}(T_*)=0.
\end{equation}
It is easy to see  the Hamiltonian is a constant, so the last equation in \eqref{el} holds for all $0\leq t\leq T$, i.e.
\begin{equation}\label{control}
|\mv-\nabla U(\my)|^2 = |\nabla U(\my)|^2.
\end{equation}
The minimum value of $\int_0^T |\mv|^2 \ud t=2(U(c)-U(a))$ is achieved when the optimal control $\mv$ in \eqref{control} takes $\mv=2 \nabla U(\my)$ before the transition state $c$. Then after passing the transition state $c$, $\my(t)$ satisfies $\dot{\my} = - \nabla U (\my)$ with $\mv=0$ and  reaches $b$ at infinite time. {\blue We remark that since $a$ and $c$ are  steady states with $\nabla U \big|_{a,c}=0$, so the optimal terminal time above is indeed $\8$. Thus the optimal controlled path $\my(t)$ can not be achieved at finite time and should be understood in the limit sense, i.e., consider the optimal control problem \eqref{oc} with starting point $\my(T_1)=a', |a'-a|\leq \eps$, ending point  $\my(T_2)=c', |c'-c|\leq \eps$ and then taking the limit $\eps\to 0$, the boundary condition in \eqref{el}, shall be replaced by $\my(T_1)\to a, \, T_1\to -\8$, $\my(T_2)\to c,\,T_2 \to +\8.$
}

Recall the transition state (saddle point) $c$ in \eqref{saddle}.
Now we verify the optimality of the construction above as follows. For any path from $A$ to $B$, it must pass a point $\my_s$ such that $U(\my_s)\geq U(c)$. Denote $T_{\my_s}$ the time passing ${\my_s}$, then from  \eqref{oc}, we have
\begin{equation}
\begin{aligned}
\int_0^T \frac12 |\mv|^2 \ud t \geq& \int_0^{T_{\my_s}} \frac12 |\mv|^2 \ud t= \frac12 \int_0^{T_{\my_s}} |\dot{\my}-\nabla U(\my)|^2+ 2\int_0^{T_{\my_s}} \dot{\my} \cdot \nabla U(\my) \ud t\\
 \geq& 2 \int _0^{T_{\my_s}} \frac{\ud }{\ud t} U(\my) \ud t = 2(U({\my_s}) -U(a)) \geq 2( U(c)-U(a)).
\end{aligned}
\end{equation}
This inequality shows that the minimal energy barrier (a.k.a. value function) is $\gamma=2(U(c)-U(a))$ and the minimal energy path must pass through $c$.
 This completes the optimal control interpretation for the MEP given in \eqref{mep}.

{\blue Motivated by this deterministic optimal control problem, the stochastic control problem in Theorem \ref{th_soc} later is the stochastic version of the  control problem \eqref{oc} with a fixed noise $\eps>0$; see \cite{fleming06} and  Remark \ref{rem_sam} after Theorem \ref{th_soc}.  For a fixed noise level $\eps>0$, we will  introduce a stopping time $\tau$, which is finite almost surely by positive recurrence.   Thus the optimal transition path can be achieved at finite time almost surely. At the finite noise level, one can directly use committor function $q$ defined in \eqref{dn} instead of the quasi-potential in \cite{Freidlin_Wentzell_2012}. Solving committor function is a linear problem while in \cite{Freidlin_Wentzell_2012}, MEP is formulated as an exit time problem computed by finding out a quasi-potential from a Hamilton-Jacobi equation.
}

\subsection{General stochatic optimal control problems with terminal cost }\label{sec_soc}
In general, one can consider a stochastic optimal control problem with some running cost function $L(t,\my_t)$ and terminal cost function $g(T, \my_T)$, where have used the convention $\my_t:=\my(t)$ in stochastic analysis.   Especially, we are interested in the optimal control for a stationary (w.r.t. time) Markov process and the terminal time being the stopping time when the SDE solution hitting some closed set $B$, i.e. $\tau:=\inf\{t\geq 0; \my_t\in \bar{B} \}$.  In this case, the terminal cost function $g$ is also called boundary cost function, to be specific in Theorem \ref{th_soc}.

Given initial  probability measure $\mu_0$ on $\nn$ concentrating on the local minimums $a$ of the potential $U$, consider  the stochastic  optimal control problem in the infinite time horizon with running cost function $\frac12|\mv(\my_t)|^2$ and boundary cost function $g(\my_\tau)$
\begin{equation}\label{soc}
\begin{array}{ll}
\gamma =\inf_{\mv } \mathbb{E}\left\{\int_{0}^{\tau} \frac{1}{2}\left|\mv\left( \my_t\right)\right|^{2} d t + \chi_{\tau<\8}g(\my_\tau) \right\} \\
\st \ud  {\my}_t=\bbs{-\nabla U\left( \my_t\right)+\mv\left(\my_t\right)} \ud t+ \sqrt{2\eps} \ud_\nn B_t, \quad t\in(0,\tau), \quad  \quad \my_t|_{t=0} \sim \mu_{0}.
\end{array}
\end{equation}
Here $\chi$ is the indicator function, $|\mv|$ is the length of $\mv$  in $\bR^\ell$ and $\ud_\nn B_t$ is a shorthand notation for the Brownian motion on manifold $\nn$ in the sense of \eqref{sde-y}, which will be used in the following context.
In Nelson's theory of stochastic mechanics, $\mv$ can be regarded as an average velocity and the running cost function  $L=\frac12 |\mv|^2$ is the classical action integrand including only kinetic energy.
The obtained optimal control $\mv(\my)$ is  called stationary Markov control policy. Since the original Markov process is on a closed manifold $\nn$, we will always has the positive recurrence property provided the landscape is continuous.   From now on, we focus on the case $\tau<\8$, a.s..

With the small parameter $\eps$, recall the original  Markov process $Y_t$ on manifold $\nn$ without control has the corresponding generator  $Qf=\eps\Delta f - \nabla U \cdot \nabla f.$  The control $\mv$ in the stochastic control problem \eqref{soc} can be regarded as an additional driven force to the original Markov process.

\section{Optimal control viewpoint for the transition path theory for  the continuous Markov process}\label{sec3}

In this section, the main goal is to give a stochastic optimal control interpretation for the transition path theory. We will first review the transition path theory in Section \ref{sec.rTPT}. Then in Section \ref{sec_soc_th}, we prove the committor function in the transition path theory leads to a stochastic optimal control with which the controlled Markov process  realizes the transitions from $A$ to $B$ almost surely. \subsection{Review of the transition path theory}\label{sec.rTPT}
Now we review and explain some concepts in the transition path theory including the committor function, the effective transition path process, and the density/current of transition paths; see original work \cite{weinan2006}.

\subsubsection{Committor function}
We  start from the original Markov process $Y_t$ with generator
\begin{equation}\label{gQ}
Qf= \eps\Delta f - \nabla U \cdot \nabla f.
\end{equation}
Let $A$ and $B$ be two disjoint absorbing sets of attractors $a, b$. To study the conditioned process with the conditions on paths starting from $A$ then ending in $B$, one should find an appropriate excessive function $q$ and calculate the transition probabilities of the conditioned process by using Doob $h$-transform via $q$.
\begin{rem}
As mentioned in Section \ref{sec_soc}, we know the SDE solution $Y_t$ hits $\overline{A\cup B}$ at a finite time due to  the positive recurrence of the process $Y_t$ on the closed manifold $\nn$. For an unbounded domain, the recurrence can be ensured by  some specific condition. {\blue For instance, we assume there exists $R_0>0$ such that for any $r>R_0$,
\begin{equation}\label{recur}
\nabla_\mathbf{n} U \Big|_{|\my|=r} > \frac{c}{r}, \quad \text{ for some constant } c>\eps d +1.
\end{equation}
Here $|\my|=r$ is the geodesic ball on $\nn$ and $\mathbf{n}\in T_\nn$ is the outer normal vector of the ball. Consequently, $\frac{|\my|^2}{2}$ is a Lyapunuov function such that
$$Q\bbs{\frac{|\my|^2}{2}} = \eps d -\nabla U \cdot \my \leq -1  \quad \text{ for }|\my|>R_0.$$   Applying \cite[Cor 2.4.2, Cor 2.4.3]{krylov} with this Lyapunov function, we know condition \eqref{recur} ensures  the existence of an invariant measure for process $Y_t$. This invariant measure is also unique by \cite[Theorem 4.1.6]{krylov}. Using the same Lyapunov function in \cite[Theorem 3.9]{Khasminskii}  (see also \cite[Theorem 3]{Veretennikov}), condition \eqref{recur} also ensures the positive recurrence in $\mathbb{R}^d$.
}
\end{rem}
Define the stopping time $\tau_B:=\inf \{t\geq 0; Y_t\in \bar{B}\}$ (resp. $\tau_A$) of process $Y_t$ when it hits $B$ (resp. $A$).
The probability for the paths hitting $B$ before $A$ is given by the forward committor function $q(\my)$, a.k.a.  harmonic potential, which is the solution of
\begin{equation}\label{dn}
Qq(\my)=0, \quad \my \in(\overline{A\cup B})^c
\end{equation}
with the Dirichlet boundary conditions
\begin{equation}\label{dbc}
q(\my)=0,\, \my\in \bar{A}, \qquad q(\my)=1,\,\my\in \bar{B}.
\end{equation}
{\blue \begin{lem}\label{lem:recur}
The solution $q(\my)$ to \eqref{dn} with \eqref{dbc} satisfies $0<q(\my)<1, \, \my\in (\overline{A\cup B})^c.$
\end{lem}
\begin{proof}
Since $\nn$ is a closed manifold, $U(\my)$ smooth enough and $\bar{A}\cap \bar{B}=\emptyset,$ there exists a solution $q(\my)\in C^2((\overline{A\cup B})^c)\cap C(\overline{A^c\cap B^c})$ to \eqref{dn} with \eqref{dbc}. Then by the strong maximum principle, we conclude $0<q(\my)<1, \, \my\in (\overline{A\cup B})^c.$
\end{proof}
}

As an important consequence, the density and the current of transition paths can be calculated using the committor function; see detailed revisit in Section \ref{sec_current_c}.

\subsubsection{Generator for the conditioned process}
To describe the conditioned process with the conditions that paths starting from $A$ then ending in $B$,
 \cite{Lu_Nolen_2015} characterized the selection of the reactive paths  coming from $A$ and then hitting $B$ by using  the probability measure on the path space  such that $\tau_A>\tau_B$.

  The associated conditioned process, called the transition path process,  is denoted as $Z_t$.  For $Z_0=\my_0\in (\overline{A\cup B})^c$, the generator of this conditioned process $Z_t$ can be  described using the Doob $h$-transform.
 Precisely, using committor function $q$ as the excessive function and by the Doob $h$-transform, the generator for conditioned process $Z_t$ is
\begin{equation}\label{Doob}
Q^q f=\frac{1}{q}Q(qf)=Qf+\frac{2\eps \nabla q}{q} \cdot \nabla f.
\end{equation}
Since $q=0$ in $\bar{A}$,  a singular drift term prevents $Z_t$ hitting $A$ and also  pushes  $Z_t \in \pt A$ into $(\overline{A\cup B})^c$.
 For the delicate case $Z_0$ starting from $\pt A$ with an appropriate initial law on $\pt A$,   \cite[Theorem  1.2]{Lu_Nolen_2015} proved that the conditioned process $Z_t$ with the augmented filtration  is same  in law as the $k$-th transition path process exiting from $A$ then hitting $B$ defined in \cite{weinan2006, SHV_2006, E_Vanden-Eijnden_2010}. More precisely, the initial and end distribution for $Z_t$, a.k.a.  reactive exit and entrance distribution $\nu_0, \nu_1$, can be calculated by the Dirichlet to Neumann map of the elliptic equation for committor function \eqref{dn} \cite[Proposition 1.5]{Lu_Nolen_2015}.

In Section \ref{sec_soc_th}, we will prove the resulting conditioned process (the transition path process) $Z_t$  can be regarded as the original process with an additional control $\mv=\frac{2\eps \nabla q}{q}$. This control, indeed optimal control, together with the original landscape $U$, leads to an effective potential $U^e:=U-2\eps \ln q$, which is $+\8$ for $\my\in \bar{A}$ and $U$ for $\my\in\bar{B}$. This effective potential guides the associated SDE of $\tilde{Y}_t$   to the absorbing set $B$ before hitting  $A$.

\subsubsection{Density and current of transition paths}\label{sec_current_c}
Next, using the conditioned process $Z_t$, which is also the controlled process $\tilde{Y}_t$ in \eqref{oc_gamma}, and its generator $Q^q$, we sketch the derivation of the current of transition paths. Recall the equilibrium density of the original Markov process $Y_t$ is $\pi\propto e^{-\frac{U}{\eps}}$. From \cite[Proposition 1.9]{Lu_Nolen_2015}\cite[Proposition 2]{weinan2006}, the density of transition paths is
 \begin{equation}\label{rhoR}
 \rho^R(\my)= \pi(\my) q(\my)(1-q(\my)).
 \end{equation}
 Then some elementary calculations show that
 \begin{equation}
 (Q^q)^* \rho^R(\my) = 0 , \quad \my \in  (\overline{A\cup B})^c .
 \end{equation}
 From \eqref{Doob} and \eqref{gQ},
 \begin{align}\label{Doob2}
0=& (Q^q)^* \rho^R = Q^* \rho^R -\nabla \cdot \bbs{\rho^R \frac{2\eps \nabla q}{q}}=\eps \nabla\cdot \bbs{\pi \nabla \frac{\rho^R}{\pi}-\rho^R \frac{2\nabla q}{q}}.
 \end{align}
 This divergence form, together with \eqref{rhoR}, gives the current of transition paths from $A$ to $B$ (upto a factor $\eps$)
 \begin{equation}\label{cc}
 J_R:=- \pi \nabla \frac{\rho^R}{\pi}+\rho^R \frac{2\nabla q}{q} = -\pi \nabla[q(1-q)] + 2\pi (1-q)\nabla q =\pi \nabla q.
 \end{equation}
 One can  directly verify there is an  equilibrium such that
 \begin{equation}\label{def:pie}
(Q^q)^* \pi^e = 0, \qquad  \pi^e := e^{-\frac{U^e}{\eps}} = \pi q^2,
 \end{equation}
which vanishes at $A$. {\blue However, we point out $\rho^R$ is not an equilibrium although $(Q^q)^* \rho^R=0$ inside  $(\overline{A\cup B})^c$. This involves a boundary measure at $\pt B$ and we will only discuss details for the discrete case in Section \ref{sec422}.
}

\subsection{Stochastic optimal control interpretation of the committor function}\label{sec_soc_th}
In this subsection, we will prove that the committor function $q$ gives a stochastic optimal control such that the controlled Markov process realizes the transition from $A$ to $B$ with minimum running and terminal cost. We will first illustrate the idea of optimal change of measure in an abstract measurable space, then prove the stochastic optimal control interpretation  in Theorem \ref{th_soc}.

\subsubsection{Duality between the relative entropy and the Helmholtz free energy}
It is well known that the canonical ensemble is closely related to the optimal change of measure for the Helmholtz free energy.  More precisely, let $(\Omega, \mathcal{F})$ be a measurable space and $\bP(\Omega)$ is the family of  probability measures on $\Omega$.
Denote Hamiltonian $H\in \bR$ as a measurable function on $\Omega$.  For a reference measure (a.k.a.  prior measure) $P\in \bP(\Omega)$ and any $\beta>0$, we define the Helmholtz free energy of  Hamiltonian $H$ with respect to   $P$ as
\begin{equation}
F(H):= - \frac1\beta  \ln \bbs{\int_\Omega e^{-\beta H(\omega)}  \ud P(\omega)}\in [-\8, +\8).
\end{equation}
For any other measure $\tilde{P}\in \bP(\Omega)$ which is absolutely  continuous w.r.t. $P$, denote $\text{KL}(\tilde{P}||P)=\int_\Omega\ln \bbs{\frac{\ud \tilde{P}}{\ud P}} \ud \tilde{P}$  the relative entropy with respect to $P$. Then we have the following Legendre-type transformations and duality in statistical mechanics; c.f. \cite{deuschel2001large}.

(i) For any measure $\tilde{P}\ll P$ 
\begin{equation}
-\frac{1}{\beta}\text{KL} (\tilde{P}||P)= \sup_H \left\{  \int_\Omega H(\omega) \ud \tilde{P}(\omega) -  F(H)   \right\};
\end{equation}

(ii) for any bounded function $H(\omega)$ 
\begin{equation}
F(H)=\inf_{\tilde{P}\ll  P} \left\{\int_\Omega H(\omega) \ud \tilde{P}(\omega) + \frac1\beta  \text{KL} (\tilde{P}||P)\right\}=\inf_{\tilde{P} \ll P} \left\{\int_\Omega \bbs{H(\omega) +\frac1\beta \ln \frac{\ud \tilde{P}(\omega)}{\ud P (\omega)} } \ud \tilde{P}(\omega)  \right\}.
\end{equation}

In Theorem \ref{th_soc} later, {\blue we will see  the second Legendre-type transformation (ii) is still true for a Hamiltonian $g$ defined in \eqref{ham}  and we use it} for probability measures (induced by coordinate processes)  on path space $\Omega=C([0,+\8); \nn)$ to prove the optimality. To show the idea of the proof, if $\ud P = \rho_0 \ud x$ and $\ud \tilde{P} =\rho \ud x$, the optimal density $\rho^*$ in the transformation is achieved when
\begin{equation}
e^{ -\beta H-1+\lambda}=\frac{\rho^*}{\rho_0}, \qquad \int \rho^* =1 = e^{\lambda-1} \int e^{-\beta H}  \rho_0 \ud x,
\end{equation}
where $\lambda$ is the Lagrange multiplier to ensure $\rho$ is a probability density.
Then we have
\begin{equation}
\int H \rho^* \ud x + \frac1\beta \int \ln \frac{\rho^*}{\rho_0} ~\rho^* \ud x = \frac{\lambda -1}\beta = F(H).
\end{equation}

\subsubsection{Committor function gives the optimal control in infinite time horizon}

{\blue In general, given an arbitrary control field
in the gradient form, $\mv = 2\eps \nabla \ln h$ for some function $h$,  the generator for the controlled process is not exactly the Doob $h$-transform. Indeed, by Ito's formula, the generator under the control $\mv$ is
\begin{equation}
Q^h= Q+\mv \cdot \nabla= Q + 2\eps \nabla \ln h \cdot \nabla.
\end{equation}
On the other hand, by elementary calculations, the Doob $h$-transform satisfies
the following identity
\begin{equation}
\frac{1}{h} Q(h f) =  Qf+ f \frac{Q h}{h} + 2\eps \nabla \ln h \cdot \nabla f = Q^h f + f \frac{Q h}{h}
\end{equation}
for any test function $f$.
We can recast the under controlled generator $Q^h$ as
\begin{equation}
Q^h f =\frac{1}{h} Q(h f) - f \frac{Q h}{h}.
\end{equation}

}

{\blue Now we give a theorem on the optimality of the control $\mv=\frac{2\eps\nabla q}{q}=2\eps\nabla \ln q$ in an infinite time stochastic optimal control problem,  where $q$ is the committor function solving \eqref{dn} with boundary condition \eqref{dbc}. As a consequence of the optimal control $\mv=2\eps\nabla \ln q$, we recover
$
\frac{1}{q} Q(q f) = Q^q  f.
$
}

\begin{thm}\label{th_soc}
Assume the original Markov process $Y_t$ has generator $Qf=\eps\Delta f -\nabla U \cdot \nabla f,$ $\eps>0$.   The forward commitor function $q$ in \eqref{dn} gives an optimal control $\mv^*=\frac{2\eps\nabla q}{q}$ in the sense that it drives the controlled process $\tilde{Y}_t$ starting from $\vec{y} \in (\overline{A\cup B}) ^c$ to the set $B$ before hitting $A$ with the least action
\begin{equation}\label{oc_gamma}
\begin{aligned}
\gamma(\vec{y}):=& \min_{\mv\in\mathcal{A}}\bE_P\left[\int_0^\tau \frac1{2} |\mv(\tilde{Y}_s)|^2 \ud s+g(\tilde{Y}_\tau)\right]\\
& \st ~ \ud \tilde{Y}_t = (-\nabla U(\tilde{Y}_t)+ \mv(\tilde{Y}_t)) \ud t + \sqrt{2\eps} \ud_\nn B_t, \,\, \tilde{Y}_0 = \vec{y},
\end{aligned}
\end{equation}
where $\tau:=\inf\{t\geq 0; \tilde{Y}_t\in \overline{A\cup B}  \}$  is the stopping time,
$
 g(x):=\left\{ \begin{array}{c}
+\8, \quad \text{ in } \bar{A},\\
0, \quad\quad\, \text{ in } \bar{B},
\end{array} \right.
$
 and the admissible control belongs to
\begin{equation}\label{AA}
\mathcal{A}:=\{\mv \in T_\nn; ~  \bE_P \bbs{e^{ \int_0^\tau \frac{1}{4\eps} |\mv(\tilde{Y}_s)|^2 \ud s }} <\8  \}.
\end{equation}
 Moreover, $\gamma(\my)=-2\eps \ln q(\my)$ and the optimal control $\mv^*$ leads to an effective potential $U^e:=U-2\eps\ln q$ for the controlled Markov process $\tilde{Y}_t$ with the generator $Q^q$ defined in \eqref{Doob}.
\end{thm}
\begin{proof}
Choose $\Omega=C([0,+\8); \nn)$, with the product $\sigma$-algebra, as our measurable space \cite{Varadhan-Stoc}. Any element $\my\in\Omega$ gives a coordinate process $Y_t(\my):= \my(t)\in \nn, \, t\geq 0.$

First, recall for a closed manifold,  the stopping time $\tau:=\inf\{t\geq 0; Y_t\in \overline{A\cup B}  \}<\8, a.s.$. From \cite[Section 6.2.1]{Evans-SDE}, the stochastic characterization of  committor function $q$ in \eqref{dn} can be expressed as
\begin{equation}\label{qq}
q(\vec{y})=\bE_P(f(Y_\tau)), \quad \vec{y}\in \overline{A\cup B}^c
\end{equation}
with  function $f(\mz)$ on $\overline{A\cup B}$, $f=0,\, \mz\in \bar{A}$ while $f=1,\,\mz\in \bar{B}.$ Here the expectation is taken under the probability measure $P$ (called reference measure) on the path space $\Omega$ associated with all realizations of the original SDE of $Y_t$ starting from $\vec{y} \in (\overline{A\cup B}) ^c$. Define $\beta := \frac1{2\eps}$ and function
 \begin{equation}\label{ham}
 g(x)= - \frac{1}{\beta}\ln f =\left\{ \begin{array}{c}
+\8, \quad \text{ in } \bar{A};\\
0, \quad\quad\, \text{ in } \bar{B}.
\end{array} \right.
\end{equation}
Then choose  Hamiltonian $H(\my):= g(Y_\tau(\my))$ for $\my\in\Omega$.
Since the original SDE for $Y_t$ is translation invariant with respect to time (stationary w.r.t. time), so the control function $\mv(Y_t)$ (stationary Markov control) does not explicitly depend on $t$ and we can take the starting time as $t=0$ without loss of generality.   We consider the optimal control problem \eqref{oc_gamma}
with the admissible control in $\mathcal{A}$.
$\mathcal{A}$ in \eqref{AA} is   the well known Novikov condition in the Girsanov's transformation, which ensures the almost sure positivity of the  Radon-Nikodym derivative in \eqref{tm58}. It is easy to see $\mv=0$ belongs to $\mathcal{A}$ so $\mathcal{A}\neq \emptyset.$ However from the definition of $g$, $g(Y_\tau)=+\8$ when $Y_\tau\in \bar{A}$, so $\mv=0$ is not an optimal control. We also remark that the selection effect due to $g$, i.e. the controlled system will hit $B$ before $A$, can be equivalently added to the admissible control set \eqref{AA}.

Second, we find the minimizer $\gamma$ (a.k.a value function) and the corresponding optimal control $\mv^*$.
Let $\tilde{P}$ be the   probability measure  on $\Omega$   associated with all realizations of the SDE of $\tilde{Y}_t$ with control $\mv$.  By the  Girsanov's theorem, $Y_t$ is a Brownian motion under measure $P$ and $\tilde{Y}_t$ is a Brownian motion under measure $\tilde{P}$. As a consequence,
\begin{equation}\label{girs}
\bE_P(e^{-\beta g(Y_\tau)})=\bE_{\tilde{P}} (e^{-\beta g(\tilde{Y}_\tau)}) = \bE_P\bbs{e^{-\beta g(\tilde{Y}_\tau)} \frac{\ud \tilde{P}}{\ud P}},
\end{equation}
where the Radon-Nikodym derivative is given by  \cite[Theorem 6.2]{Varadhan-Stoc}
\begin{equation}\label{tm58}
\frac{\ud \tilde{P}}{\ud P}= e^{- \beta \int_0^\tau \sqrt{2\eps}  \mv(\tilde{Y}_s) \cdot\ud_\nn B_s -\frac\beta{2} \int_0^\tau |\mv(\tilde{Y}_s)|^2 \ud s} >0, \quad P\text{-a.s.}
\end{equation}
and the sign in front of $\ud_\nn B$ is negative following the convention. In short, we have
\begin{equation}\label{changeP}
\bE_P(e^{-\beta H(\my)})=\bE_P(e^{-\beta g(Y_\tau)}) = \bE_P\bbs{e^{-\beta\big(g(\tilde{Y}_\tau) +\int_0^\tau \sqrt{2\eps}  \mv(\tilde{Y}_s) \cdot \ud_\nn B_s +\frac1{2} \int_0^\tau |\mv(\tilde{Y}_s)|^2 \ud s \big) }}.
\end{equation}
Then by Jensen's inequality and \eqref{girs},
\begin{equation}\label{jensen}
e^{- \beta \bE_P \bbs{g(\tilde{Y}_\tau) +  \int_0^\tau \frac1{2} |\mv(\tilde{Y}_s)|^2 \ud s}} \leq \bE_P \bbs{e^{-\beta g(\tilde{Y}_\tau)} \frac{\ud \tilde{P}}{\ud P}} = \bE_P(e^{-\beta g(Y_\tau)}).
\end{equation}
Here the equality is achieved if and only if $g(\tilde{Y}_\tau) +\int_0^\tau \sqrt{2\eps} \mv(\tilde{Y}_s)\cdot \ud_\nn B_s
+  \int_0^\tau \frac1{2} |\mv(\tilde{Y}_s)|^2 \ud s$ is deterministic.
{\blue Using Lemma \ref{lem:recur} and \eqref{qq}, we know
\begin{equation}\label{eq:q0}
q(\vec{y})=\bE_P(f(Y_\tau))= \bE_P(e^{-\beta g(Y_\tau)})>0, \quad \vec{y}\in (\overline{A\cup B})^c.
\end{equation}
Thus the RHS of \eqref{jensen} is always positive. Taking the logarithm to both sides, we have
\begin{equation}\label{tm322}
-\beta \bE_P \bbs{g(\tilde{Y}_\tau) +  \int_0^\tau \frac1{2} |\mv(\tilde{Y}_s)|^2 \ud s} \leq  \ln  \bbs{\bE_P \bbs{e^{-\beta g(\tilde{Y}_\tau)} \frac{\ud \tilde{P}}{\ud P}}}=\ln  \bbs{\bE_P(e^{-\beta g(Y_\tau)})}.
\end{equation}
Notice if the LHS of \eqref{jensen} is zero, \eqref{tm322} still holds since $-\8$ is always smaller than a finite number.
}
Therefore, we obtain
\begin{equation}
 \bE_P \bbs{g(\tilde{Y}_\tau) +  \int_0^\tau \frac1{2} |\mv(\tilde{Y}_s)|^2 \ud s} \geq - \frac1\beta \ln  \bbs{\bE_P \bbs{e^{-\beta g(\tilde{Y}_\tau)} \frac{\ud \tilde{P}}{\ud P}}}= - \frac1\beta \ln\bbs{\bE_P(e^{-\beta g(Y_\tau)}) },
\end{equation}
 which, together with  \eqref{eq:q0}, gives
\begin{equation}
\gamma(\vec{y}) \geq  - \frac1\beta \ln\bbs{\bE_P(e^{-\beta g(Y_\tau)}) }=- \frac1\beta \ln q(\vec{y}), \quad \vec{y}\in (\overline{A\cup B})^c.
\end{equation}
Furthermore, the verification theorem \cite[IV.3, Theorem 5.1]{fleming06} shows that the equality is indeed  achieved
$\gamma(\vec{y}) = -\frac1\beta \ln q(\vec{y})$. Actually, $\gamma(\vec{y})$ satisfies the Hamilton-Jacobi-Bellman (HJB) equation
\begin{equation}\label{HJB}
\eps\Delta \gamma -\frac{1}{2}|\nabla \gamma|^2 -\nabla U \cdot \nabla\gamma=0, \,\, \text{ in }(\overline{A\cup B})^c, \qquad \gamma= g \,\, \text{ on }{\overline{A\cup B}}.
\end{equation}
 The associated optimal control (optimal feedback), such that
\begin{equation}\label{I}
 \tilde{H}(\my):=g(\tilde{Y}_\tau) +\int_0^\tau \sqrt{2\eps} \mv^*(\tilde{Y}_s) \cdot \ud_\nn B_s+  \int_0^\tau \frac1{2}  |\mv^*(\tilde{Y}_s)|^2 \ud s
\end{equation}
 being  deterministic, is given by
\begin{equation}
\mv^* = - \nabla \gamma =\frac1\beta  \nabla \ln \bE_P\bbs{e^{-\beta g(Y_\tau)}} = \frac1\beta \nabla \ln q.
\end{equation}
{
Indeed, one can verify this by applying Ito's formula to $\gamma(\tilde{Y}_{ \tau})$
\begin{equation}
\begin{aligned}
I_1:=&\gamma(\tilde{Y}_{\tau})-\gamma(\tilde{Y}_0)=1-\gamma(\vec{y})\\
 =& \int_0^{\tau} \tilde{Q} \gamma(\tilde{Y}_s)\ud s +\sqrt{2\eps} \int_0^{\tau}  \nabla\gamma\cdot \ud_\nn B_s\\
=&\int_0^{\tau} \eps \Delta \gamma(\tilde{Y}_s) +\bbs{-\nabla U +\mv} \cdot \nabla\gamma(\tilde{Y}_s)\ud s + \sqrt{2\eps} \int_0^{\tau}  \nabla\gamma\cdot\ud_\nn B_s\\
=&\int_0^{\tau} \bbs{\eps\Delta \gamma-\nabla U \cdot \nabla\gamma -|\nabla \gamma|^2 }(\tilde{Y}_s)\ud s +\sqrt{2\eps} \int_0^{\tau}  \nabla\gamma \cdot\ud_\nn B_s\\
=&-\frac{1}{2}\int_0^{\tau} |\nabla \gamma|^2 (\tilde{Y}_s)\ud s +\sqrt{2\eps} \int_0^{\tau}  \nabla\gamma\cdot\ud_\nn B_s,
\end{aligned}
\end{equation}
where $\tilde{Q}$ is the generator of $\tilde{Y}_t$ and we used $\mv^*=- \nabla \gamma$ and \eqref{HJB}. On the other hand,
\begin{equation}
\begin{aligned}
I_2:=& \int_0^{\tau} \sqrt{2\eps} \mv^*(\tilde{Y}_s) \cdot\ud_\nn B_s+  \int_0^{\tau} \frac1{2}  |\mv^*(\tilde{Y}_s)|^2 \ud s\\
=&- \sqrt{2\eps}\int_0^{\tau} \nabla \gamma \cdot\ud_\nn B_s +\frac{1}{2} \int_0^{\tau}  |\nabla \gamma|^2 \ud s
\end{aligned}
\end{equation}
due to $\mv^*=- \nabla \gamma.$
Then we have
\begin{equation}
I_1+I_2 =  \int_0^{\tau} \bbs{     \eps \Delta \gamma -\frac12 |\nabla \gamma|^2 - \nabla U \cdot  \nabla\gamma }\ud s=0,
\end{equation}
which, together with $I_1=1-\gamma(\vec{y})$, shows $\tilde{H}$ is deterministic.
}
Since the optimality ensures $\tilde{H}$ in \eqref{I} is deterministic, so
$$\bE_P\bbs{e^{-\beta \tilde{H}}}=e^{-\bE_P(\beta \tilde{H})}=e^{\bE_P(\beta I_1)}=e^{\beta(1-\gamma)}=q e^\beta<+\8$$
 and thus $\mv^*\in \mathcal{A}$.

 Finally, plugging the optimal control $\mv^*= 2\eps \nabla \ln q$ into \eqref{oc_gamma} shows the effective potential for the controlled Markov process $\tilde{Y}_t$ is
 \begin{equation}\label{effectiveU}
 U^e=U-2\eps\ln q.
 \end{equation}
 Then by Ito's formula, the master equation for the controlled Markov process $\tilde{Y}_t$ is
 \begin{equation}\label{master_oc}
  \pt_t \rho = \eps \nabla \cdot  \left( e^{-\frac{U^e}{\eps}} \nabla (\rho e^{\frac{U^e}{\eps}}) \right) = \eps \nabla \bbs{\pi^e \nabla \frac{\rho}{\pi^e}}.
 \end{equation}
 {\blue Here $\pi^e$ is the effective equilibrium defined in \eqref{def:pie}.}
\end{proof}
{\blue
\begin{rem}
Although we focus on a reversible process defined by SDE \eqref{sde-y}, we remark  Theorem \ref{th_soc} still holds for irreversible process with a general drift $\vec{b}$\footnote{Under assumptions for existence of a solution, for instance, $\vec{b}$ is smooth enough and $\vec{b}(\vec{y})\cdot \vec{y}\leq c(1+|\vec{y}|^2)$. } and the generator $\tilde{Q}f=\eps \Delta f + \vec{b}\cdot \nabla f$. The  optimal control will be given by $2\eps \nabla \ln q$ with $q$ being the solution to $\tilde{Q}q=0$ and \eqref{dbc}.
However, the Fokker-Planck equation for irreversible process does not have a relative entropy formulation. We refer to \cite{GL21} for a structure preserving numerical scheme for  irreversible processes with a general drift field, which can be leveraged to construct an optimally controlled random walk on point clouds for general irreversible processes.
\end{rem}
\begin{rem}
Notice  the value function $\gamma(\vec{y})=-2\eps \ln q(\vec{y})$ satisfies
\begin{equation}
 0 = \eps\Delta \gamma -\frac{1}{2}|\nabla \gamma|^2 -\nabla U \cdot \nabla\gamma= -e^{\frac{\gamma}{2\eps}} Q(e^{-\frac{\gamma}{2\eps}} ).
\end{equation}
 We comment on the connection with the Logarithmic transformation framework developed by  Sheu-Fleming \cite{Sheu85} \cite[Section VI]{fleming06}.   Define a Hamiltonian operator $Hf := e^f Q(e^{-f})$, then by \cite[Lemma 9.1, p 257]{fleming06}
\begin{equation}
H f =\min_h \{-Q^h f  -  Q^h(\ln h) + \frac{Q h}{h}\}.
\end{equation}
This Hamiltonian operator and the associated HJB for the optimal control for exit problem has been studied by \cite{Hartmann2016}.  Particularly, the authors applied the optimal control verification theorem in \cite[Section VI]{fleming06} to the exit problem in the infinite time horizon and constructed an optimally controlled Markov chain based on the solution to the associated HJB.
\end{rem}
}
\begin{rem}\label{rem_sam}
We remark Theorem \ref{th_soc} uses a terminal cost function $g$ and the associated committor function $q$ to construct an optimal feedback $\mv^*$ and thus the associated controlled process $\tilde{Y}$. More importantly, the transition from absorbing set $A$ to another absorbing set $B$ is a rare event for the original Markov process $Y_t$ while this transition becomes an almost sure event for the controlled Markov process $\tilde{Y}_t$. This has significant statistic advantage because using the controlled process $\tilde{Y}_t$, which realizes the conformational transitions almost surely. The computations for statistic quantities in the original rare event becomes more efficient; see the Algorithm \ref{alg:RW} for the controlled random walk on point clouds.
\end{rem}

\section{Markov process and transition path theory on point clouds}\label{sec4}
This section focuses on constructing an approximated Markov process on point clouds and computing the  discrete analogies in the transition path theory on point clouds. In Section \ref{sec_upwind}, we will first introduce a finite volume scheme which approximates the original Langevin dynamics \eqref{sde-y} and its master equation on $\nn$.  In Section \ref{sec_soc_D}, based on the approximated Markov process, we design the discrete counterparts for the committor functions, the discrete Doob $h$-transform, the generator for the optimal controlled Markov process in the transition path theory on point clouds.
\subsection{Finite volume scheme and the approximated Markov process on point clouds}\label{sec_upwind}
In this section, we first propose a finite volume scheme for the original Fokker-Planck equation based on a  data-driven approximated Voronoi tesselation  for $\nn$. Then we reformulate it as a Markov  process on point clouds.
\subsubsection{Voronoi tesselation and finite volume scheme}
Suppose $(\nn, d_\nn)$ is a $d$ dimensional smooth closed submanifold of $\mathbb{R}^\ell$ and $d_{\nn}$ is  induced by the Euclidean metric in $\bR^\ell$. $D:=\{\my_i\}_{i=1:n} $ are point clouds sampled from some density on $\nn$ bounded below and above. It is proved that the data points $D$ are well-distributed on $\nn$  whenever the points are sampled from a density function with lower and upper bounds \cite{Dejan15, liu2019rate}. Define the Voronoi cell as
\begin{equation}
C_i:= \{\my\in \nn ; \ud_\nn(\my,\my_i)\leq \ud_\nn(\my,\my_j) \text{ for all }\my_j\in D\} \quad  \text{ with volume } |C_i|=\hs^d(C_i).
\end{equation}
Then $\nn=\cup_{i=1}^n  C_i$ is a Voronoi tessellation of  $\nn$. Denote the Voronoi face for cell $C_i$ as
\begin{equation}
\Gamma_{ij}:= C_i\cap C_j   \text{ with its area  } |\Gamma_{ij}|=\hs^{d-1}(\Gamma_{ij}),
\end{equation}
for any $j=1, \cdots, n$. If $\Gamma_{ij}= \emptyset$ or $i =  j$ then we set $|\Gamma_{ij}|=0$.
Define the associated adjacent sample points as
\begin{equation}
VF(i):=\{j; ~\Gamma_{ij}\neq \emptyset\}.
\end{equation}

By Ito's formula,  SDE \eqref{sde-y} gives the following Fokker-Planck equation, which is the master equation for the density $\rho(\my)$ in terms of $\my$,
\begin{equation}\label{FP-N}
\pt_t \rho = \nabla \cdot  (\eps \nabla \rho + \rho \nabla U)=: FP^{\nn} \rho.
\end{equation}
Denote the equilibrium  $\pi:= e^{-\frac{U}{\eps}}$.  The Fokker-Planck operator has the following equivalent form
\begin{equation}\label{FPequi}
\begin{aligned}
FP^\nn(\rho) =& \eps \Delta\rho + \nabla \cdot (\rho \nabla U) = \nabla \cdot \bbs{\rho (\eps \nabla \ln \rho + \nabla U)}
\\
=& \eps \nabla \cdot \bbs{\rho \nabla \ln \frac{\rho}{\pi}}= \eps \nabla \cdot \bbs{\pi \nabla \frac{\rho}{\pi}}.
\end{aligned}
\end{equation}

Using \eqref{FPequi}, we have the following finite volume  scheme
\begin{equation}\label{upwind}
\frac{\ud}{\ud t}\rho_i { |C_i|}
=   \sum_{j\in VF(i)}
\frac{\pi_i+ \pi_j}{2|y_i-y_j|}
{ |\Gamma_{ij}| }
\left( \frac{\rho_j}{\pi_j}- \frac{\rho_i}{\pi_i} \right), \quad i=1, \cdots,n,
\end{equation}
where $\pi_i$ is the approximated equilibrium density  at $\my_i$ with the volume element $|C_i|$.
One can also recast \eqref{upwind} as a backward equation formulation
\begin{equation}\label{Bupwind}
\frac{\ud}{\ud t}\frac{\rho_i}{\pi_i}
=  \sum_{j\in VF(i)}
\frac{\pi_i+ \pi_j}{2\pi_i |C_i||y_i-y_j|}
{ |\Gamma_{ij}| }
\left( \frac{\rho_j}{\pi_j}- \frac{\rho_i}{\pi_i} \right) , \quad i=1, \cdots,n.
\end{equation}
{\blue  Define a
 stochastic matrix $Q$   as
\begin{equation}\label{asQ}
Q_{ij} =
\frac{\pi_i+ \pi_j}{2\pi_i |C_i||y_i-y_j|}
{ |\Gamma_{ij}| }\geq 0, \quad j\neq i,
\qquad Q_{ii}= - \sum_{j\neq i} Q_{ij}.
\end{equation}
Notice that  row sums zero, $\sum_{j} Q_{ij} = 0$. Then $Q$ is the generator of the associated Markov process.
We rewire \eqref{Bupwind} in the matrix form
\begin{equation}
\frac{\ud}{\ud t}\frac{\rho}{\pi} = Q  \frac{\rho}{\pi} \qquad \text{ with } \frac{\rho}{\pi}=\bbs{\frac{\rho_i}{\pi_i}}_{i=1:n}.
\end{equation}
}
With  an adjoint $Q$-matrix,  \eqref{upwind} can also be recast in a matrix form
\begin{equation}\label{tm4.9}
\frac{\ud}{\ud t} \rho|C| = Q^* (\rho|C|) \qquad \text{ with } \rho|C|=\{ \rho_i |C_i|\}_{i=1:n}.
\end{equation}

In practice, since we don't have the exact manifold information, the volume of the Voronoi cells $C_k$ and the area of the Voronoi faces $\Gamma_{kl}$ need to be approximated.  We refer to \cite[Algorithm 1]{GLW20} for the algorithm for approximating $|C_k|$ and $|\Gamma_{ij}|$ and the convergence analysis of this solver \eqref{upwind} for the Fokker-Planck equation \eqref{FP-N}.  We denote the approximated volumes as $|\tilde{C}_k|$ and the approximated areas as $|\tilde{\Gamma}_{kl}|$.
After replacing $C_k$ by the approximated volumes $|\tilde{C}_k|$ and replacing $\Gamma_{kl}$ by the approximated areas $|\tilde{\Gamma}_{kl}|$, \eqref{upwind}/ \eqref{mp1} becomes an approximated Markov process on point clouds, which is an implementable solver for the Fokker-Planck equation on $\nn$.
 We drop tildes without confusion in the following contexts.
\subsubsection{Markov process on point clouds} With the approximated volumes  $|C_i|$ and the approximated areas  $|\Gamma_{ij}|$,
one can interpret the finite volume scheme \eqref{upwind} as
 the forward equation for a Markov process with
 transition probability $P_{ji}$ (from $j$ to $i$) and jump rate $\lambda_j$
\begin{equation}\label{mp1}
\frac{\ud}{\ud t}\rho_i |C_i| = \sum_{j\in VF(i)} \lambda_j P_{ji} \rho_j |C_j| - \lambda_i \rho_i |C_i|,\quad i=1, 2, \cdots, n,
\end{equation}
where
\begin{equation}\label{def59}
\begin{aligned}
\lambda_i := \sum_{j\neq i} Q_{ij}= \frac1{2|C_i|\pi_i}\sum_{j\in VF(i)} \frac{\pi_i+ \pi_j}{|y_i-y_j|}|\Gamma_{ij}|, \quad i=1, 2, \cdots, n; \\
 P_{ij}:=\frac{Q_{ij}}{\lambda_i}=\frac{1}{\lambda_i}\frac{\pi_i+ \pi_j}{2\pi_i |C_i|}\frac{|\Gamma_{ij}|}{|y_i-y_j|}, \quad j\in VF(i); \quad P_{ij}=0, \quad j\notin VF(i).
 \end{aligned}
\end{equation}
Assume $\pi_i>0$ for all $i$, then we have $\lambda_i>0$ for all $i$.
One can see
it
satisfies $\sum_i P_{ji}=1$ and the detailed balance property
\begin{equation}\label{db}
 Q_{ji}\pi_j |C_j|=P_{ji} \lambda_j\pi_j |C_j| =  P_{ij} \lambda_i\pi_i |C_i| = Q_{ij} \pi_i |C_i| .
\end{equation}

\subsection{Committor function, currents, and  controlled Markov process on point clouds}\label{sec_soc_D}
In this section, we first review the corresponding concepts for the transition path theory on point clouds. Then from the optimal control viewpoint, we construct a finite volume scheme for the controlled Fokker-Planck equation and  the  associated controlled Markov process (random walk on point clouds).
\subsubsection{Committor function, currents, and transition rate}
Suppose the local minimums $a, b$ of $U$ are two cell center  with index $i_a$ and $i_b$. Below, we clarify the discrete counterparts of Section \ref{sec.rTPT} for committor functions $q$,  the density of transition paths $\rho^R$, the current of transition paths $J_R$ and the transition rates $k_{AB}$.

 First, from the backward equation formulation \eqref{Bupwind}, the forward committor function $q_i, \, i=1, \cdots,n$ from $a$ to $b$ satisfies
\begin{equation}\label{Dcom}
\begin{aligned}
\sum_{j\in VF(i)}
Q_{ij}
\left( q_j- q_i \right)&=0, \quad  i\neq i_a, i_b,\\
&q_{i_a}=0, \quad q_{i_b}=1.
\end{aligned}
\end{equation}

Second, the discrete density of the reactive path \cite[Remark 2.10]{MMS2009} is defined as
\begin{equation}\label{dis_equi}
\rho_i^R := \pi_i q_i (1-q_i).
\end{equation}

Third,
with the constructed $Q$-matrix in \eqref{asQ}, the current from site $i$ to site $j$ of the reactive path from state $a$ to state $b$  is given by \cite[Remark 2.17]{MMS2009}
\begin{equation}\label{current}
J_{ij}^R := Q_{ij}\pi_i (q_j-q_i)=\piw \frac{q_j-q_i}{|y_i-y_j|},
\end{equation}
which is the counterpart of the current in \eqref{cc}. Due to \eqref{Dcom}, it is easy to check the current is divergence free, i.e., satisfying the Kirchhoff's current law,
\begin{equation}\label{div}
\sum_{j\in VF(i)} J_{ij}^R=0, \quad i\neq i_a, i_b.
\end{equation}

Finally, the transition rate from absorbing set $A$ to  $B$ can be calculated from the current. It is shown in \cite[Theorem 2.15]{MMS2009} that the transition rate from $A$ to $B$ is given by
\begin{equation}\label{eq:kab}
  k_{AB} = \sum_{i\in  A}\sum_{j\in VF(i)}J^R_{ij}.
\end{equation}
Particularly, if there is only one point $\my_{i_a}\in A$, then
\begin{equation}
k_{AB} = \la Q q , \delta_{i_a}  \ra_\pi = \sum_{j\in VF(i_a)} J_{i_a, j}^R,
\end{equation}
where $\delta_{i_a}$ is the Kronecker delta with value $1$ if $i=i_a$ while $0$ otherwise.

\subsubsection{Finite volume scheme and $Q^q$-matrix for controlled Markov process}\label{sec422}

Similar to the controlled Markov process in \eqref{master_oc}, we give the controlled Markov process on point clouds below. Suppose the local minimums $a, b$ of $U$ are two cell center  with index $i_a$, $i_b$ and  for simplicity we assume there is only one point $\my_{i_a} \in \bar{A}$. We construct a controlled random walk on point clouds $\{\my_i\}_{i=1:n, i\neq i_a}$. The controlled random walk for the general case that more than one point belong to $A$ is similar.
{\blue Below, we  derive the master equation for this controlled random walk and still denote the density at states $\{\my_i\}_{i=1:n, i\neq i_a}$ as $\rho_i.$}

First, with the effective potential
$U^e$ in \eqref{effectiveU}, the effective equilibrium is
$
\pi^e= e^{-\frac{U^e}{\eps}}= q^2 \pi.
$
{\blue We now construct a Markov process with a $Q^q$-matrix on states $\{\my_i\}_{i=1:n, i\neq i_a}$  such that
\begin{enumerate}[(i)]
\item $\pi^e= q^2 \pi$ is an equilibrium;
\item it satisfies the detailed balance $|C_i|\pi^e_i Q_{ij}^q = |C_j|\pi_j^e Q^q_{ji};$
\item it satisfies mass conservation $\frac{\ud}{\ud t}\sum_i\rho_i { |C_i|}=0.$
\end{enumerate}
}
Plug $\pi^e_i := q_i^2 \pi_i $ into the scheme \eqref{upwind}.
We propose a  finite volume scheme for the controlled Fokker-Planck equation \eqref{master_oc}
\begin{equation}\label{upwind_oc}
\begin{aligned}
\frac{\ud}{\ud t}\rho_i { |C_i|}
= \sum_{j\in VF(i), j\neq i_a}
\frac{q_i q_j (\pi_i+ \pi_j)}{2|y_i-y_j|}
{ |\Gamma_{ij}| }
\left( \frac{\rho_j}{q_j^2\pi_j}- \frac{\rho_i}{q_i^2 \pi_i} \right),
\quad i=1, \cdots,n,\,\, i\neq i_a.
\end{aligned}
\end{equation}

{\blue
With $Q$  in \eqref{asQ}, we define  for $i=1, \cdots, n, i\neq i_a,$
\begin{equation}\label{Qq}
Q^q_{ij}:=\frac{q_j}{q_i} Q_{ij}\geq 0 \quad \  (i\neq j),\ \text{and} \ Q^q_{ii}= -\sum_{j\neq i, i_a} Q^q_{ij}.
\end{equation}
$Q^q$ is an $n-1$ by $n-1$  stochastic matrix which has zero row sum, i.e. $\sum_{j\neq i_a} Q^q_{ij}=0$,  $i\neq i_a$.
One can  recast \eqref{upwind_oc} as a backward equation
\begin{equation}\label{upwind_oc_b}
\begin{aligned}
\frac{\ud}{\ud t}\frac{\rho_i} { \pi_i^e}
=&  \sum_{j\in VF(i), j\neq i_a}\frac{q_j}{q_i}
\frac{ (\pi_i+ \pi_j)}{2\pi_i|C_i| |y_i-y_j|}
{ |\Gamma_{ij}| }
\left( \frac{\rho_j}{\pi_j^e}- \frac{\rho_i}{ \pi^e_i} \right)\\
=& \sum_{j\in VF(i), j\neq i_a} Q^q_{ij} \left( \frac{\rho_j}{\pi_j^e}- \frac{\rho_i}{ \pi^e_i} \right),
\quad i=1, \cdots,n,\, i\neq i_a
\end{aligned}
\end{equation}
and $Q^q$ is the effective generator for the controlled Markov process on point clouds $\{\my_i\}_{i=1:n, i\neq i_a}$.

For (i), by plugging $\pi^e=q^2 \pi$ into \eqref{upwind_oc}, one obtains $\pi^e$ is an equilibrium solution.

For (ii), we can verify for $i,j \neq i_a$ and $i\neq j$
\begin{equation}\label{dbq}
|C_i|\pi^e_i Q^q_{ij} = q_i q_j \pi_i Q_{ij}  |C_i|= q_i q_j \pi_j Q_{ji} |C_j|= |C_j|\pi^e_j Q^q_{ji},
\end{equation}
where we used the detailed balance property of $Q$.

For (iii), recast \eqref{upwind_oc} as matrix form
\begin{equation}
\begin{aligned}
\frac{\ud}{\ud t} \rho_i |C_i| =& \sum_{j\in VF(i), j\neq i_a} Q^q_{ij} \pi_i^e|C_i|\left( \frac{\rho_j}{\pi_j^e}- \frac{\rho_i}{ \pi^e_i} \right)\\
=& \sum_{j\in VF(i), j\neq i_a} \bbs{Q^q_{ji} |C_j|\rho_j- Q^q_{ij} |C_i| \rho_i }.
\end{aligned}
\end{equation}
The summation w.r.t. $i$ for both sides concludes
mass conservation $\frac{\ud}{\ud t}\sum_{i\neq i_a}\rho_i { |C_i|}=0.$

Second,
 we plug $\rho^R_i$ defined in \eqref{dis_equi}  into  \eqref{upwind_oc} and  by using \eqref{Dcom}, we have
\begin{equation}\label{upwind_equi}
\begin{aligned}
 \sum_{j\in VF(i), j\neq i_a}
\frac{q_i q_j ( \pi_i+ \pi_j)}{2|y_i-y_j|}
{ |\Gamma_{ij}| }
\left( \frac{1}{q_j}- \frac{1}{q_i} \right)=-  \sum_{j\in VF(i), j\neq i_a}
\frac{\pi_i+ \pi_j}{2|y_i-y_j|}
{ |\Gamma_{ij}| }
\left( q_i-q_j\right)=0,\\
 \quad i=1, \cdots,n,\,\, {\blue i\neq i_a, i_b.}
\end{aligned}
\end{equation}
{\blue However, we emphasize that $\rho^R$ is not an equilibrium  for the proposed Markov process \eqref{upwind_oc} because
$Q_{i_b i_b}^q \neq Q_{i_b i_b}.$

Third, let us discuss the spectral gap of $Q^q$-matrix:
\begin{enumerate}[(1)]
\item From zero row sum property, we know $0$ is an eigenvalue of $Q^q$.
\item From the detailed balance property \eqref{dbq},  we know the dissipation estimate
\begin{equation}
\la Q^q u, u\ra_{\pi^e|C|} = -\frac{1}{2} \sum_{i,j; i\neq j} Q^q_{ij}(u_j-u_i)^2 \pi^e_i |C_i|\leq 0.
\end{equation}
Thus  the eigenvalues of $Q^q$ satisfy $0=\lambda_1 \geq \lambda_2\geq \cdots.$
\item Since the manifold $\nn$ is assumed to be connected, so the associated graph for $Q^q$ given by  Delaunay triangulation is connected. Hence $\la Q^q u, u\ra_{\pi^e|C|} =0$ if and only if $u\equiv \text{constant}$. Therefore  there is a spectral gap for $Q^q$, i.e., $0=\lambda_1 > \lambda_2\geq \cdots.$
\end{enumerate}

}

 }

Finally,  one can recast \eqref{upwind_oc_b} as  the controlled Markov process with the controlled transition probability $P^q_{ji}$ (from site $j$ to $i$) and the controlled jump rate $\lambda_j^q$
\begin{equation}\label{mp_oc}
\frac{\ud}{\ud t}\rho_i |C_i| = \sum_{j\in VF(i), j\neq i_a} \lambda^q_j P^q_{ji} \rho_j |C_j| - \lambda^q_i \rho_i |C_i|,\quad i=1, 2, \cdots, n,\,\, i\neq i_a,
\end{equation}
where
\begin{equation}
\begin{aligned}
\lambda^q_i =\sum_{j\neq i} Q^q_{ij}= \frac1{2|C_i|\pi_i}\sum_{j\in VF(i), j\neq i_a} \frac{q_j}{q_i}\frac{\pi_i+ \pi_j}{|y_i-y_j|}|\Gamma_{ij}|, \quad i=1, 2, \cdots, n,  \,\, i\neq i_a; \\
 P^q_{ij} = \frac{Q^q_{ij}}{\lambda^q_i} = \frac{1}{\lambda_i^q}\frac{q_j}{q_i}\frac{\pi_i+ \pi_j}{2\pi_i |C_i|}\frac{|\Gamma_{ij}|}{|y_i-y_j|}, \quad j\in VF(i), j\neq i_a; \quad P^q_{ij}=0, \quad j\notin VF(i).
\end{aligned}
\end{equation}
We remark that the controlled generator $Q^q$ \eqref{Qq} constructed using the Doob $h$-transform  is also used in \cite{Todorov, Hartmann2016} for the exit problem for controlled Markov chains. 

 Under the constructed controlled transition probability $P_{ij}^q$ for the controlled random walk, the transition from the  metastable state $a\in A$ to $b\in B$  is almost sure in $O(1)$ time rather than a rare event. Taking advantage of this nature, we will provide an algorithm for finding the mean transition path from $A$ to $B$. This algorithm can be efficiently implemented by Monte Carlo simulation of the controlled random walk on point cloud. See Algorithm \ref{alg:meanpath}.

\begin{rem}
Similar to the continuous version, we formally calculate the discrete optimal control fields below.
From \eqref{upwind_oc},
\begin{equation}\label{QQQ}
\begin{aligned}
\frac{\ud}{\ud t}\rho_i { |C_i|}
&= (Q^q)^* (\rho|C|):=  \sum_{j\in VF(i)}
\frac{ (\pi_i+ \pi_j)}{2|y_i-y_j|}
{ |\Gamma_{ij}| }
\left( \frac{q_i}{q_j} \frac{\rho_j}{\pi_j}- \frac{q_j}{q_i} \frac{\rho_i}{ \pi_i} \right)\\
&= \sum_{j\in VF(i)}
\frac{ (\pi_i+ \pi_j)}{2|y_i-y_j|}
{ |\Gamma_{ij}| }
\left(  \frac{\rho_j}{\pi_j}- \frac{\rho_i}{ \pi_i} + \bbs{\frac{q_i}{q_j} -1} \frac{\rho_j}{\pi_j}  - \bbs{ \frac{q_j}{q_i}-1} \frac{\rho_i}{ \pi_i}  \right)\\
&=  \sum_{j\in VF(i)}
\frac{ (\pi_i+ \pi_j)}{2|y_i-y_j|}
{ |\Gamma_{ij}| }
\left(  \frac{\rho_j}{\pi_j}- \frac{\rho_i}{ \pi_i} \right) +  \sum_{j\in VF(i)}
\frac{ (\pi_i+ \pi_j)}{2|y_i-y_j|}
{ |\Gamma_{ij}| }
\bbs{\frac{\rho_j}{q_j\pi_j}  + \frac{\rho_i}{ q_i\pi_i} }  \bbs{q_i-q_j}  \\
&= Q^* (\rho |C|)  -   \sum_{j\in VF(i)} |\Gamma_{ij}| \frac{q_j-q_i}{|y_i-y_j|}\frac{ (\pi_i+ \pi_j)}{2}\bbs{\frac{\rho_j}{q_j\pi_j}  + \frac{\rho_i}{ q_i\pi_i} }\\
&= Q^* (\rho |C|)  -   \sum_{j\in VF(i)} |\Gamma_{ij}|v_{ij} \rho_{ij},
\end{aligned}
\end{equation}
where $v_{ij} = 2 \frac{(q_j-q_i)}{|y_j-y_i|} \frac{2}{q_i+q_j}$ and  $\rho_{ij} =\frac18 (q_i+q_j)   (\pi_i+ \pi_j)\bbs{\frac{\rho_j}{q_j\pi_j}  + \frac{\rho_i}{ q_i\pi_i} }$.
Thus, as a counterpart of  Theorem \ref{th_soc}, from the optimal control viewpoint for
the Markov process (random walk) on point clouds, we can regard
$
v_{ij} = 2 \frac{(q_j-q_i)}{|y_j-y_i|} \frac{2}{q_i+q_j}
$
as the discrete optimal feedback control field from $i$ to $j$ (along edge $e_{ij}$ of the associated Delaunay triangulation).
\end{rem}

\section{Data-driven solver and computations}\label{sec_com}
In this section,  we introduce  the algorithms for finding the transition path on point clouds.  As we will see, the transition from one metastable state to another  for the optimally controlled Markov process is no longer rare. We can efficiently simulate these transition events and compute the mean transition path based on a level set determined by the committor function $q$ and a mean path  iteration algorithm on point clouds adapted from a finite temperature string method in  \cite{FTSM}. Algorithms for the construction of the approximated Markov chain and the dominant transition path  are given in Section \ref{sec_dom}. Algorithms for the Monte Carlo simulation and the mean transition path based on the controlled Markov process are given in Section \ref{sec_MC}.

\subsection{Computation of dominant transition path}\label{sec_dom}

We first need to construct the approximated Markov chain based on the  point clouds $\{\my_i\}_{i=1:n}$, i.e., to  compute the coefficients in the discrete generator \eqref{asQ}. Particularly, the approximated cell volumes $|C_k|$ and the approximated edge areas $|\Gamma_{kl}|$ can be obtained by the approximated Voronoi tessellation in \cite[Algorithm 1]{GLW20}.  Another related  local meshed method for computing the committor function via point clouds was given in \cite{Lu18c}.

Then based on the associated Markov process \eqref{asQ} with the approximated coefficients, we can compute the dominant transition path and the transition rate between metastable states on manifold $\nn$. Below we will simply mention the basic concepts and algorithms of the transition path theory of the Markov jump process  for completeness. Further details can be referred to \cite{MMS2009}.

We seek the dominant transition path from the starting state  $A$ to the ending state $B$. All  algorithms presented are  also valid for any starting state in absorbing set $A$ and ending state in set $B$. This dominant path defined in \cite{MMS2009} is the reactive path connecting $A$ and $B$ that carries the most probability current. We construct a weighted directed graph $G(V, E)$ using dataset $V=\{\my_i\}_{i=1:n}$ as  nodes, $E=\{e_{ij},\, J^R_{ij}>0\}$ as a directed edge with weight $J^R_{ij}$. Here  $J^R_{ij}>0$ is computed via \eqref{current}. From  \eqref{current}, there is no loop in the directed graph $G(V,E)$.

Given the starting and ending states $A, B\subset\{\my_i\}_{i=1:n}$, a reactive trajectory from $A$ to $B$ is an ordered sequence $P=[\my_0, \my_1,... \my_k], \,\my_i\in V,\, (\my_i, \my_{i+1})\in E$ such that $\my_0\in A$, $\my_k\in B$ and $\my_i\in (A\cup B)^c$, $0<i<k$ for some $k\leq n$.  We denote the set of all such reactive trajectories by $\Ps$. From \eqref{current}, along any reactive trajectory $P\in \Ps$, the values of the committor function
\begin{equation}\label{increase}
0=q_0<q_1<\cdots<q_k=1
\end{equation}
is strictly increasing from $0$ to $1$.
Given a reactive trajectory $P$, the maximum current carried by this reactive trajectory $P$, called capacity of $P$, is
$$c(P):=\min_{(i,j)\in P}J^R_{ij}.$$
Among all possible trajectories from $A$ to $B$, one can further find the one with the largest capacity
\begin{equation}
c_{\max}:= \max_{P\in \Ps} c(P), \quad P_{\max} \in \argmax_{P\in \Ps} c(P).
\end{equation}
We call the associated edge
\begin{equation}\label{eq:bn}
  (b_1, b_2)=\argmin_{(i,j)\in P_{\max}} J^R_{ij}
\end{equation}
the dynamical bottleneck with the weight $c_{\max}= J_{b_1  b_2}^R$.
For simplicity, we assume $J_{ij}^R$ are distinct, so $b_1, b_2$ are uniquely determined.

Finding the bottleneck provides a divide-and-conquer algorithm for finding the most probable path recursively.
The dominant transition path is the reactive path with the largest effective probability current \cite{MMS2009, E_Vanden-Eijnden_2010}.
Computing the dominant transition path is a recursion of finding the maximum capacity
on subgraphs.

Now we use the bottleneck  $(b_1, b_2)$ and level-set of committor function $q$ to divide the original graph $G(V,E)$ into two disconnected subgraphs $G_L$ and $G_R$ as below.

Note that every path in $P_{\max}$ pass through the bottleneck  $(b_1, b_2)$. Thus the weight of each edge in $P_{\max}$ is larger than the weight of bottleneck $J_{b_1 b_2}^R$. So we  first remove all the edges of the original graph $G(V,E)$ with weight smaller than $J_{b_1 b_2}^R$.  Denote
\begin{equation}\label{q53}
V_L:= \{\my_i; \ q_i\le q_{b_1}\}, \quad  V_R:= \{\my_i; \ q_i\ge  q_{b_2}\}.
\end{equation}
Construct the new graph
\begin{equation}\label{eq:GL}
\begin{aligned}
G_L:=(V_L, E_L),\quad\text{with}\ E_L:=\{e_{ij}\in E; \,\my_i, \my_j \in V_L,\, J^R_{ij}> J_{b_1 b_2}^R\};
\\
G_R:=(V_R, E_R),\quad\text{with}\ E_R:=\{e_{ij}\in E;\, \my_i, \my_j \in V_R,\, J^R_{ij}> J_{b_1 b_2}^R\}.
\end{aligned}
\end{equation}
Then we find the dominant transition path in $G_L$ from $A$ to $b_1$ and in $G_R$ from $b_2$ to $B$. So the computation of the dominant transition path is simply finding the bottleneck recursively.


In summary, we will first compute the committor function $q$ by solving the linear system \eqref{Dcom}. Then we construct the graph $G(V,E)$, and compute the dominant transition path based on recursively finding the bottlenecks  and the dominant transition paths, see \cite{MMS2009} for further implementation details of the algorithmic constructions.

\subsection{Mean transition path and the computation on point clouds}\label{sec_MC}
 
The dominant transition path from metastable state $A$ to $B$ obtained by TPT is a transition path that carries the most probability current. Below we will introduce the concept mean transition path by taking expectation with respect to the transition path density \eqref{rhoR}, which forms the rationale of our algorithm.

For any co-dimension one surface $S$ on $\mathcal{N}\subset \mathbb{R}^\ell$, we define its projected mean
\begin{equation}\label{eq:pinS}
  \vec{p}_S := \Pc \Big(Z_S^{-1}\int_{S}x\pi(x)q(x)(1-q(x))\ud\sigma\Big),
\end{equation}
where $Z_S=\int_{S}\pi(x)q(x)(1-q(x))\ud\sigma$ is the normalization constant, and $\Pc: \mathbb{R}^\ell\rightarrow \mathcal{N}$ is a projection, e.g. the closest point projection,  which is assumed unique in our paper.
We denote the mean transition path by $\vec{p}(\alpha)\in\mathcal{N}\subset\mathbb{R}^\ell$, where $\alpha\in [0,1]$ is the normalized arc length parameter that $\abs{\vec{p}'(\alpha)}\equiv Const$.

First, notice from \eqref{increase}, the committor function $q$ strictly increases from 0 to 1 along the transition path $\vec{p}$ from $A$ to $B$. We  assume the manifold $\nn$ can be parameterized by $(q,\sigma)$. Second,
 choose $S$ in \eqref{eq:pinS} as the iso-committor surface intersecting $\vec{p}(\alpha)$:
\begin{equation}\label{eq:isoq}
S_\alpha=\{x\in\mathcal{N}|q(x)=q(\vec{p}(\alpha))\}.
\end{equation}
We define $\vec{p}(\alpha)$ as the projected mean on the iso-committor surface $S_\alpha$. By the coarea formula on manifold, for any $\alpha\in[0,1]$, we can rewrite \eqref{eq:pinS} as
\begin{equation}\label{eq:pinisoq}
  \vec{p}(\alpha) = \Pc\Big(Z_{\alpha}^{-1}\int_{\mathcal{N}}x\pi(x)|\nabla q(x)|\delta(q(x)-q(\vec{p}(\alpha)))\ud x\Big),
\end{equation}
where we used $q(x)$ is constant on  $S_\alpha$ and is included in the normalization constant $Z_{\alpha}^{-1}:=Z_{S_\alpha}^{-1}q(\vec{p}(\alpha))(1-q(\vec{p}(\alpha)))$. We denote \eqref{eq:pinisoq} as a projected average $\vec{p}(\alpha) = \Pc\left< x \right>_{\pi_{q(\vec{p}(\alpha))}}$, where the average is taken with respect to the density $\pi_{q(\vec{p}(\alpha))}(x) \propto \pi(x)|\nabla q(x)|\delta(q(x)-q(\vec{p}(\alpha)))$ on $\mathcal{N}$.

Note in \eqref{eq:pinisoq}, the definition of $\vec{p}(\alpha)$ depends on $\vec{p}(\alpha)$ itself, we can compute $\vec{p}$ by a Picard iteration, i.e.,
\begin{equation}\label{eq:picard}
  \vec{p}^{l+1}(\alpha) = \Pc\Big(Z_{\alpha}^{-1}\int x\pi(x)|\nabla q|\delta(q(x)-q(\vec{p}^l(\alpha)))\ud x\Big):=\Pc\left< x \right>_{\pi_{q(\vec{p}^l(\alpha))}}.
\end{equation}
This resembles the finite temperature string method \cite{FTSM, Ren2005Transition}, which is developed to compute the average of RHS by sampling techniques. However, since the transition between metastable states rarely happens, the sampling is difficult.

With the help of the controlled dynamics dictated by effective potential $U^e$, one can compute the mean transition path $\vec{p}(\alpha)$ efficiently. Note that the optimally controlled equilibrium is only a modification of $\pi$ with a prefactor, i.e.,  $\pi^e=C q^2 \pi $, where constant $C$ ensures $\int_{\mathcal{N}}\pi^e\ud x=1$. Therefore, since $q(x)$ is constant on $S_\alpha$,  mean transition path  $\vec{p}(\alpha)$ can be identically recast as
\begin{equation}\label{eq:pinisoq2}
\begin{aligned}
  \vec{p}(\alpha)   &=\Pc\Big((Z_{\alpha}Cq^2(\vec{p}(\alpha)))^{-1}\int_{\mathcal{N}}x\pi^e(x)|\nabla q|\delta(q(x)-q(\vec{p}(\alpha)))\ud x \Big)=: \Pc \left< x \right>_{\pi^e_{q(\vec{p}(\alpha))}}.
\end{aligned}
\end{equation}
The density $\pi^e_{q(\vec{p}(\alpha))}\propto \pi^e(x)|\nabla q|\delta(q(x)-q(\vec{p}(\alpha)))$. The mean transition path can be computed by Picard iteration
\begin{equation}\label{eq:picardq}
  \vec{p}^{l+1}(\alpha) = \Pc \left< x \right>_{\pi^e_{q(\vec{p}^l(\alpha))}}.
\end{equation}
Under the dynamics governed by $U^e$, the exit from the attraction basin of metastable state $A$ is almost sure in $O(1)$ time, thus the sampling of the transition is much easier.

On the numerical aspect, we can also compute $\vec{p}$ on point clouds.  Given a point cloud $D=\{\my_i\}_{i=1:n}$, we simulate a random walk $\{\my^q_t\}$ on $D$ based on the controlled generator $Q^q$ in \eqref{Qq}. In details, we first extend the Markov process with $(Q^q)_{ij}$ in \eqref{Qq} to include the site $i_a$. Then we have $\lambda^q_{i_a} = +\8$,  so the waiting time at $i_a$ is zero. Thus at $t=0$, we start the simulation  at  $\my^q_0\in  VF(i_a)$ with probability
\begin{equation}\label{exitA}
\begin{aligned}
  P^q_{i_a j} = \frac{q_j(\pi_{i_a} + \pi_{j})}{\mathcal{Z}}\frac{|\Gamma_{j i_a}|}{|y_{i_a} - y_j|} ,\quad  j\in VF(i_a),
  \qquad  \mathcal{Z} =\sum_{j\in VF(i_a)} q_j\frac{\pi_{i_a}+ \pi_j}{|y_{i_a}-y_j|}|\Gamma_{i_a j}|.
\end{aligned}
\end{equation}
In  other words, $P^q_{i_a j} \propto J_{i_a j}^R$. We  refer to \cite[Lemma 1.3]{Lu_Nolen_2015} for the reactive exit distribution  on $\pt A$ for the continuous Markov process.
Suppose $\my^q_{t_k}=\my_i$, the next step is to update $\Delta t_k$ and $\my^q_{t_{k+1}}$ as follows.  (i) The waiting time $\Delta t_k = t_{k+1}-t_k \sim \mathcal{E}(\lambda^q_i)$ is an exponentially distributed random variable with rate $\lambda^q_i$; (ii) $\my^q_{t_{k}}$ jumps to $\my_j\in VF(\my_i)$ with probability $P_{ij}^q\equiv Q_{ij}^q/\lambda^q_i$, where $\lambda^q_i$ is defined in \eqref{def59}.
We repeat this simulation $K$ times to obtain the data $\{\my^q_{k}, \Delta t_k\}_{k=0:K}$,  in which we restart the simulation from $A$ each time when we hit $B$. Denote a sampled trajectory part $P_r$ of length $r$, from $\pt A$ to $B$, as $P_r:=\{(\my^q_0, \Delta t_0), (\my^q_1, \Delta t_1), \cdots, (\my^q_r, \Delta t_{r})\}$ such that $\my^q_0 \in VF(i_a)$ and $\my^q_r\in B$. We summarize this simulation in Algorithm \ref{alg:contr}.

\begin{algorithm}[t] \label{alg:contr}
\SetAlgoLined
\Parameter{Maximum iteration $K_{\rm max}$.}
 Set $k=0$, $r=0$. Generate $\my^q_{0}\in VF(i_a)$ with probability $P_{i_a j}^q$.

 $\my^q_{{k+1}}:= \my_{t_{k+1}}^q=\my_s$, where $s=\min\{s|\sum_{j=1, j\neq i}^{s}P^q_{ij}\geq \eta\}$, where $\eta\sim U[0,1]$ is a uniformly distributed random variable.

 $t_{k+1}=t_k+\Delta t_k$,  $\Delta t_k$ being an exponentially distributed random variable with rate $\lambda^q_i$.

 $k\leftarrow k+1$, $r\leftarrow r+1$. Repeat until $\my^q_r\in B$. Record the trajectory $P_r=\{(\my^q_0, \Delta t_0),  \cdots, (\my^q_r, \Delta t_r)\}.$

Reset $r=0$, $\my^q_{r}\in VF(i_a)$ according to Step 1. Repeat the above iterations until  $k$ exceeds the maximum iteration number $K_{\rm max}$.
\caption{Algorithm for controlled random walk on point clouds}\label{alg:RW}
\end{algorithm}

To implement the Picard iteration \eqref{eq:picardq} using data set $\{\my^q_{k}, \Delta t_k\}_{k=0:K}$, we need to approximate the density $\pi^e_{q(\vec{p}(\alpha))}$ at first. We make an assumption that $\pi^e_{q(\vec{p}(\alpha))}$ on $S_\alpha$ is localized in $\mathcal{B}^{\mathbb{R}^\ell}_{r_0}(\vec{p}(\alpha))$,  the neighborhood of $\vec{p}(\alpha)$ with radius $r_0$ in Euclidean space $\mathbb{R}^\ell$, and $|\nabla q|$ is approximately constant in $\mathcal{B}^{\mathbb{R}^\ell}_{r_0}(\vec{p}(\alpha))$. Indeed, similar assumption was also taken in the construction of finite temperature string method. With this assumption, $\left< x \right>_{\pi^e_{q(\vec{p}(\alpha))}}\approx \left< x \right>_{\tilde{\pi}^e_{q(\vec{p}(\alpha))}}$, where the density $\tilde{\pi}^e_{q(\vec{p}(\alpha))}(x)\propto \pi(x)\chi_{\mathcal{B}^{\mathbb{R}^\ell}_{r_0}(\vec{p}(\alpha))}(x)$. Taking advantage of ergodicity, we get
$$\left< x \right>_{\pi^e_{q(\vec{p}(\alpha))}}\approx \frac{1}{T_\alpha}\sum_{k=0}^{K}\my^q_k\chi_{\mathcal{B}^{\mathbb{R}^\ell}_{r_0}(\vec{p}(\alpha))}(\my^q_k)\Delta t_k,$$
where $T_\alpha=\sum_{k=0}^{K}\Delta t_k\chi_{\mathcal{B}^{\mathbb{R}^\ell}_{r_0}(\vec{p}(\alpha))}(\my^q_k)$.

Numerically, we discretize $\vec{p}(\alpha), \alpha\in [0,1]$ into $P=\{p_m\}_{m=1:M}$ for some $M\in\mathbb{N}$. For $l$-th iteration step and for any $p_m^l$, we select segments of reactive trajectories inside the ball $\mathcal{B}^{\mathbb{R}^\ell}_{r_0}(p^l_m)$, where the radius $r_0>0$ is chosen such that  $\{\my^q_k\}\cap \mathcal{B}^{\mathbb{R}^\ell}_{r_0}(p^l_m)$ has enough samples. Denote the resulting samples as
 \begin{equation}\label{alg510}
  \{\my^q_k\}_{k=0:K}\cap \mathcal{B}^{\mathbb{R}^\ell}_{r_0}(p^l_m)=\{\my^q_{r_1},\dots \my^q_{r_s}\}, \quad r_1,r_2,\dots, r_s\in\{0,1,\dots, K\},
\end{equation}
and the Picard iteration before projection takes the form
 \begin{equation}\label{alg511}
 \tilde{p}^{l+1}_m := \frac{1}{\Delta T_l}\sum_{j=1}^{s}\my^q_{r_j}\Delta t_{r_j}, \qquad \Delta T_l = \sum_{j=1}^{s}\Delta t_{r_j}.
 \end{equation}

Furthermore,  in order to avoid the issue that all $M$ discrete points overlap and concentrate on few ones, an arc-length reparameterizing procedure similar to \cite{string} is needed.

 To do the reparameterization, we first compute
\begin{equation}\label{alg512}
  S_1=0, \quad S_m=\sum_{j=2}^{m}|\tilde{p}_{m}^{l+1}-\tilde{p}_{m-1}^{l+1}|, \quad m=2,\dots, M.
\end{equation}
 Then the total length of $\tilde{P}^{l+1}$ is approximately $S_M$. We do the arc-length reparameterizations by linear interpolation as follows. (i) Denote $L_m: = \frac{m-1}{M-1} S_M$,  $m=1,2,\dots,M$; (ii) find the index $m'$
such that $S_{m'} \leq L_m < S_{m' +1}$; (iii) calculate the linear interpolation
\begin{equation}\label{alg513}
\hat{p}_m^{l+1} \approx \frac{L_m - S_{m'}}{S_{m'+1} - S_{m'}} \tilde{p}_{m'}^{l+1} + \frac{S_{m' +1} - L_m}{S_{m'+1} - S_{m'}} \tilde{p}_{m'+1}^{l+1}.
\end{equation}
Since we don't explicitly know the manifold, the projection step is done by updating $p_m^{l+1}$ as the nearest point of $\hat{p}_m^{l+1}$ in data set $\{\my_k^q\}_{k=0:K}$.
Then we can obtain the new path $P^{l+1}=\{p_m^{l+1}\}_{m=1:M}$. This updating process can be iteratively repeated until convergence, i.e., $P^{l+1} = P^{l}$ up to some tolerance. We summarize the above algorithm for finding mean transition path in Algorithm \ref{alg:meanpath}.

 Note that the Algorithm \ref{alg:meanpath} only uses the local neighbours of each $p_m^l$ in the data set $\{\my_k^q\}$. In contrast, the algorithm to finding dominant transition path, which is revisited in Section 5.1, must consider the entire graph $G(V,E)$ with all of the nodes $\{\my_k\}$. So the proposed algorithm may be more efficient when the data set $\{\my_k^q\}$ is very large and most of the points are far away from the optimal transition path.

\begin{algorithm}[t] \label{alg:meanpath}
\SetAlgoLined
\Parameter{Simulation data $\{\my^q_k\}_{k=0:K}$, waiting time $\{\Delta t_k\}_{k=0:K}$, radius $r_0>0$.}

 Set $l=0$ and for some $M$, initialize a discrete path $P^l=\{p^l_m\}_{m=1:M}$  on manifold $\nn$ connecting $A$ and $B$, where $p^l_m\in\nn$.

 For every $1\leq m\leq M$, collect all sample points in $\mathcal{B}^{\mathbb{R}^\ell}_{r_0}(p^l_m)$ based on \eqref{alg510}.

 Update path $\tilde{P}^{l+1}$ with the projected average for $m=1,\ldots, M$ via \eqref{alg511}.

 Compute $S_1, \ldots, S_M$ via \eqref{alg512}.

 Compute $\hat{p}_m^{l+1}=p_m^{l+1}$ by arc-length reparameterization \eqref{alg513}.

 Updating $P^{l+1}=\{p_m^{l+1}\}$ by finding the nearest point of each $\hat{p}_m^{l+1}$ in $\{\my_k^q\}_{k=0:K}$

 $l \leftarrow l+1$. Repeat until $P^l$ converges or $l$ exceeds a prescribed number $L_{\max}$.
\caption{Finding mean transition path on point clouds generated by controlled random walk.}
\end{algorithm}

\section{Numerical results}\label{sec_num}
{\blue In this section, based on the   mean transition path algorithms in Section \ref{sec_MC}, we conduct three examples including two examples of Muller potential and a real world example for an alanine dipeptide   with a full atomic molecular dynamics  data.}
\subsection{Synthetic examples of Mueller potential}
We choose the Mueller potential on $\mathbb{R}^2$ as the illustrative example and map it to different manifolds. The Mueller potential on $\mathbb{R}^2$ is
\begin{equation}\label{eq:Mueller}
  U(X,Y) := \sum_{i=1}^{4}A_i\exp\left(a_i(X-\alpha_i)^2 + b_i(X-\alpha_i)(Y-\beta_i) + c_i(Y-\beta_i)^2\right).
\end{equation}
The parameters are set to be $A_{1-4}=-2,-1,-1.7,0.15$, $a_{1-4} = -1,-1,-6.5, 0.7$, $b_{1-4} = 0,0,11, 0.6$, $c_{1-4} = -10,-10,-6.5, 0.7$, $\alpha_{1-4} = 1,0,-0.5,-1$, $\beta_{1-4} = 0,0.5,1.5,1$. Denote $\vec{X}=(X,Y)$. This potential has three local minima $\vec{X}_1, \vec{X}_2, \vec{X}_3$ and two saddle points $\vec{X}_4, \vec{X}_5$. The contour plot of the Mueller potential and the stationary points are shown in Fig.~\ref{fig:Mueller2D}.

We are interested in the transitions from the metastable state $\vec{X}_1$ to $\vec{X}_3$. Finding the transition path from $\vec{X}_1$ to $\vec{X}_3$ is a well-studied problem. One can compute the transition path by some existing methods like string method \cite{string}, etc., which is  shown in Fig.~\ref{fig:Mueller2D}.
\begin{figure}[htbp]
  \centering
  \includegraphics[width=0.5\textwidth]{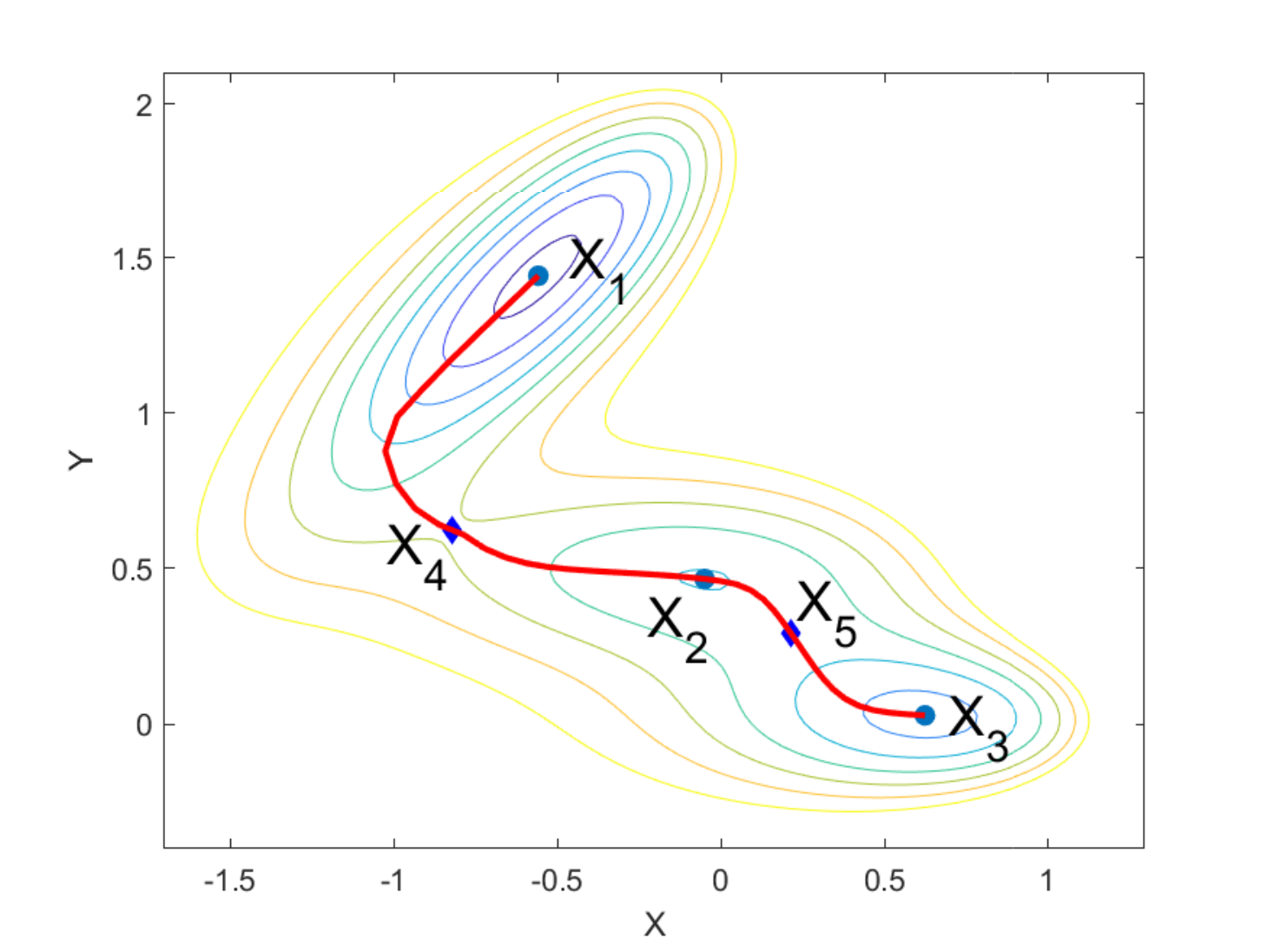}
  \caption{Contour plot of 2D Mueller potential $U(X,Y)$ in \eqref{eq:Mueller} and transition path from $\vec{X}_1$ to $\vec{X}_3$. The blue dots are local minima $\vec{X}_{1,2,3}$. The blue diamonds are saddle points $\vec{X}_{4,5}$. The red line is the transition path obtained by string method. }\label{fig:Mueller2D}
\end{figure}

\paragraph{\bf  Example 1: Mueller potential on sphere.} We map the Mueller potential to $\nn=\mathbb{S}^2$ by the stereographic projection $X=x/(1-z),\, Y=y/(1-z)$.  For any point $(x,y,z)\in \mathbb{S}^2$ except the north pole, we define $U_{\nn}(x,y,z)$ on $\mathbb{S}^2$ as
\[U_{\nn}(x,y,z) = U_{\mathbb{S}^2}(x,y,z) = U\left(\frac{x}{1-z}, \frac{y}{1-z}\right), \quad (x,y,z)\in \mathbb{S}^2\]
and consider the transitions between two metastable states under the dynamics \eqref{sde-y}. It is easy to obtain that the invariant distribution of $\my_t$ is $\pi(\my)\propto \exp(-\eps ^{-1} U_{\mathbb{S}^2}(\my))$, $\my= (x,y,z)$.
We then generate  the data set $D=\{\my_i\}_{i=1:4000}$  uniformly on $\mathbb{S}^2$ and set $\pi_i = \exp(-\eps ^{-1} U_{\mathbb{S}^2}(\my_i))$, respectively. We choose the starting state $A = D\cap \mathcal{B}^{\mathbb{R}^3}_{0.05}(\vec{X}_1)$ and the ending state $B = D \cap \mathcal{B}^{\mathbb{R}^3}_{0.05}(\vec{X}_3)$, where $\mathcal{B}^{\mathbb{R}^3}_r(\vec{x})=\{\my\in\mathbb{R}^3| |\my-\vec{x}|<r\}$ is the ball centered at $\vec{x}$ with radius $r$ in $\mathbb{R}^3$. With this data set, we can compute the committor function $q(\my)$ by solving the approximated Voronoi tesselation and the linear system \eqref{Dcom}. Since it is a diagonally dominant system, the solution is unique and we utilize a diagonal preconditioning trick to make the computation more effective and stable.

The effective potential $U^e$ with $\eps =0.1$ is shown in Fig. \ref{fig:control}. Under the controlled random walk \eqref{mp_oc}, the transition from $A$ to $B$ happens much easier.  One can see from the Monte Carlo simulation in Fig. \ref{fig:control} (c),  the exit from the attraction basin of metastable state $\vec{X}_1$ is almost sure rather than a rare event.   Compared with the original $U_{\mathbb{S}^2}(\my)$, in the effective potential, the local minima at $A$ disappears and $U^e$ tends to infinity when approaching $A$; see Fig. ~ \ref{fig:control}(b). We can also find that the dominant transition path almost goes along the gradient direction of $U^e$ from $A$ to $B$.  Taking the maximum iteration $K_{\rm max}=10^5$ in  the  Monte Carlo simulation Algorithm \ref{alg:contr},  we find 48  transition trajectories from $A$ to $B$; see Fig. \ref{fig:control}(c). In the simulation with the uncontrolled generator $Q$, there is no transition from $A$ to $B$ at all in $10^5$ steps.
The mean transition path based on Algorithm \ref{alg:meanpath} is also shown using the solid black line in Fig. \ref{fig:control}(c). We set $M=100$ and $L_{\max}=20$ in the algorithm.  The mean transition path derived by Monte Carlo simulation data (the solid black line in Fig. \ref{fig:control}(c)) highly coincides with the dominant transition path in TPT (red circles in Fig. \ref{fig:control}(c)). {\blue Providing rigorous justification for this remarkable coincidence is an important problem for future study.  Indeed, both dominant transition path algorithm and mean transition path algorithm are designed to find  ``ensembles of transition paths'' for fixed noise level $\eps>0$. Moreover, they both   rely on the level-set of committor function $q$ to order point clouds; see \eqref{q53} and \eqref{eq:isoq}.  }

\begin{figure}[t]
\includegraphics[angle = 90, width = 1\textwidth]{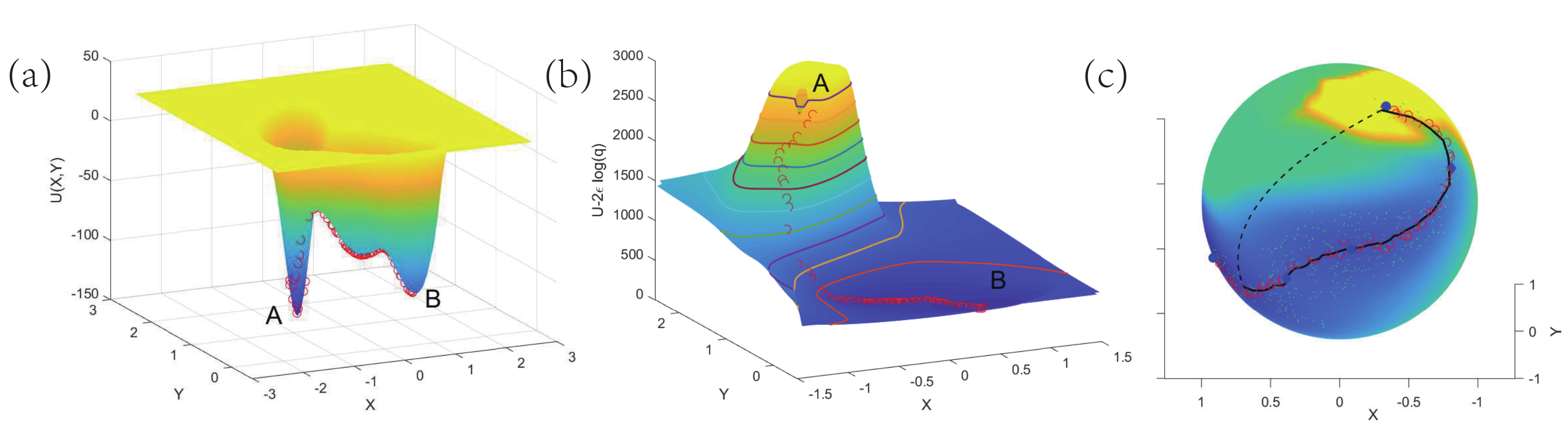}
\caption{Projections of the Mueller potential and the effective potential on $\mathbb{S}^2$ to $\mathbb{R}^2$. (a) the projection of $U_{\mathbb{S}^2}(x,y,z)$ to $\mathbb{R}^2$. (b) the projection of effective potential $U^e=U_{\mathbb{S}^2}-2\eps\log q$ to $\mathbb{R}^2$. The contour lines of $U^e$ are shown in colored lines.  The hole at $A$ is due to $U^e(A) = +\infty$. (c) The Monte Carlo simulation result of transition from $A$ to $B$ based on effective backward operator $Q^q$ and the mean transition path. The background is heat plot of $U^e$. The green dots are Monte Carlo samples. The black line is the mean transition path computed by Algorithm \ref{alg:meanpath} while the dashed line is the initial discrete path. In all sub-figures, the dominant transition paths are shown with red circles.}\label{fig:control}
\end{figure}

With committor function $q(\my)$, we can obtain the dominant transition path by applying the TPT algorithms.
We show the results for different $\eps $ in Fig.~\ref{fig:MuellerSphere}(a)-(c). As a comparison, we also compute the minimum energy path in the limit $\eps \to 0$. This can be done by minimizing the Freidlin-Wentzell action functional \cite{Freidlin_Wentzell_2012}. Namely, it is the solution of the following variational problem
\begin{equation}\label{eq:quasipotential}
  S(B;A) = \inf_{T>0}\inf_{\vec{x}\in A,\vec{y}\in B}\inf_{\psi(t)\subset \nn:\psi(0)=\vec{x}, \psi(T)=\vec{y}}\int_{0}^{T}\norm{\dot{\psi}+\nabla_{\nn}U(\psi)}^2 dt.
\end{equation}
This problem can be efficiently solved  by minimum action method (MAM) on manifold \cite{MAM, zhou2016}.  Note that $U_{\mathbb{S}^2}(x,y,z)$ can be naturally extended to $\mathbb{R}^3\backslash\{z=1\}$, one can directly apply MAM on $\nn$ by a properly designed MAM on $\mathbb{R}^3$. This zero-noise path is used as a reference; see solid red line in  Fig.~\ref{fig:MuellerSphere}.  We also map the transition path on $\nn$ to $\mathbb{R}^2$ by the stereographic projection. The result is shown in Fig.~\ref{fig:MuellerSphere}(d)-(f).

\begin{figure}[htbp]
  \centering
  \includegraphics[angle = 90, width=1\textwidth]{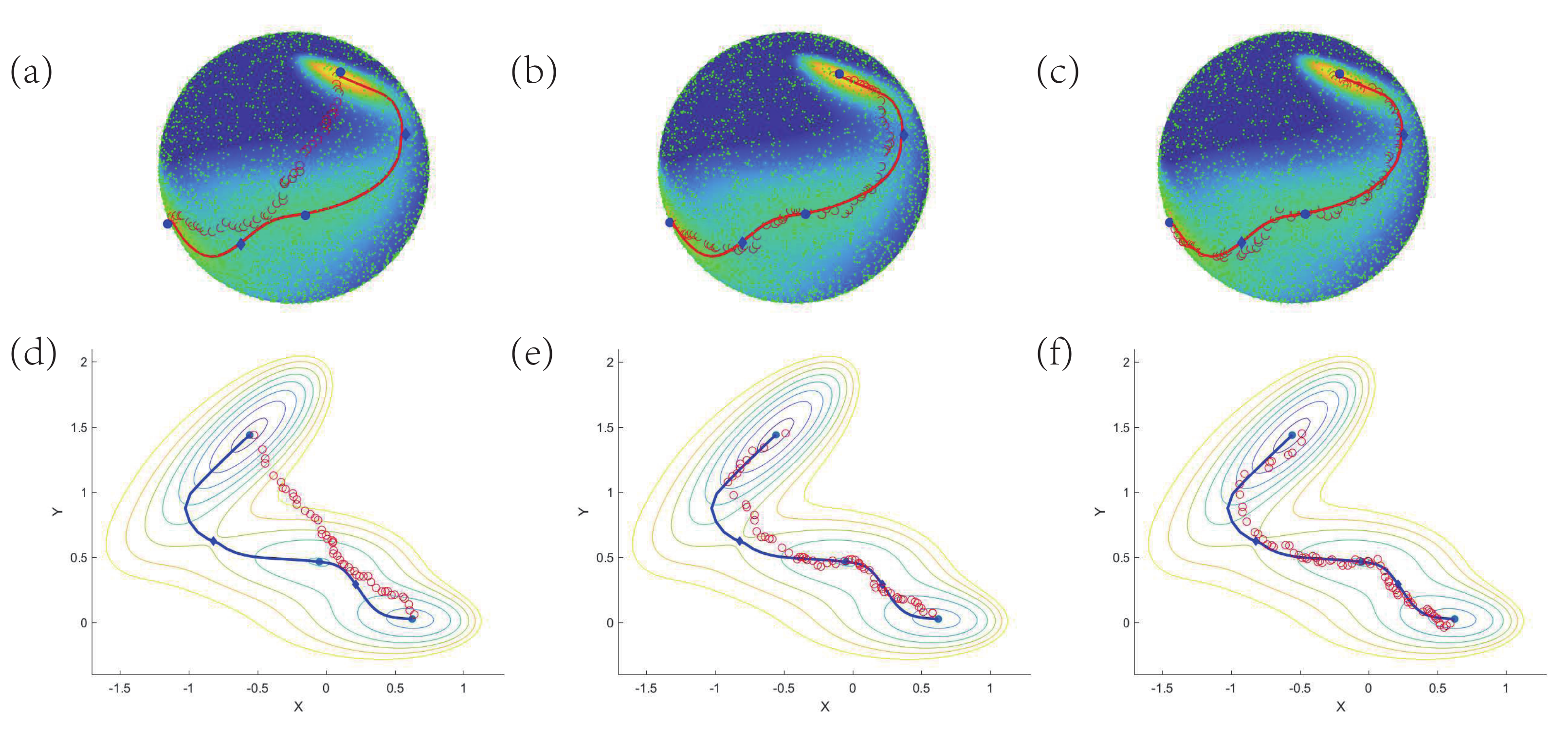}
  \caption{The dominant transition paths from $\vec{X}_1$ to $\vec{X}_3$ on $\mathbb{S}^2$ for different $\eps $. (a)(d)   $\eps =1$. (b)(e) $\eps =0.2$. (c)(f) $\eps =0.05$. In all sub-figures, the blue dots are metastable states $\vec{X}_{1,2,3}$ while the blue diamonds are saddle points. (a)-(c) The background of each sub-figure shows the heat plot of $U_{\mathbb{S}^2}(x,y,z)$. The small green dots are 4000 random samples. The paths of red circles are dominant transition paths obtained by TPT. The zero-noise minimum energy paths computed by MAM are shown with solid red  lines. (d)-(f) The background is the contour plot of potential $U(X,Y)$. The solid blue lines and red circles are the projections of transition paths in (a)-(c), respectively.
   }\label{fig:MuellerSphere}
\end{figure}

In Fig.~\ref{fig:MuellerSphere}, we find that as $\eps $ tends to zero, the dominant transition path converges to the zero-noise path obtained by MAM both on manifold $\mathbb{S}^2$ and the 2D projection. This is consistent with the Freidlin-Wentzell theory. The results are stable when different random samples are utilized.

We can find some critical transition states along the dominant transition path with the help of probability current $J^R$. Note that finding the dominant transition path is a divide-and-conquer algorithm by finding a sequence of dynamical bottlenecks. The key transition states must have the least current $J^R$. In Fig.~\ref{fig:BN}, we plot the current $J^R$ along the dominant transition path. The states with the least five $J^R$ are marked and projected to $\mathbb{R}^2$, $\eps $ is chosen to be $0.1$. One can see that all these five states are in neighbourhoods of saddle points or local minima. As stated by the Freidlin-Wentzell theory, the transition path in the zero noise limit must pass through stationary points, which is confirmed in our computations.

\begin{figure}[htbp]
  \centering
  \includegraphics[angle = 90, width=1\textwidth]{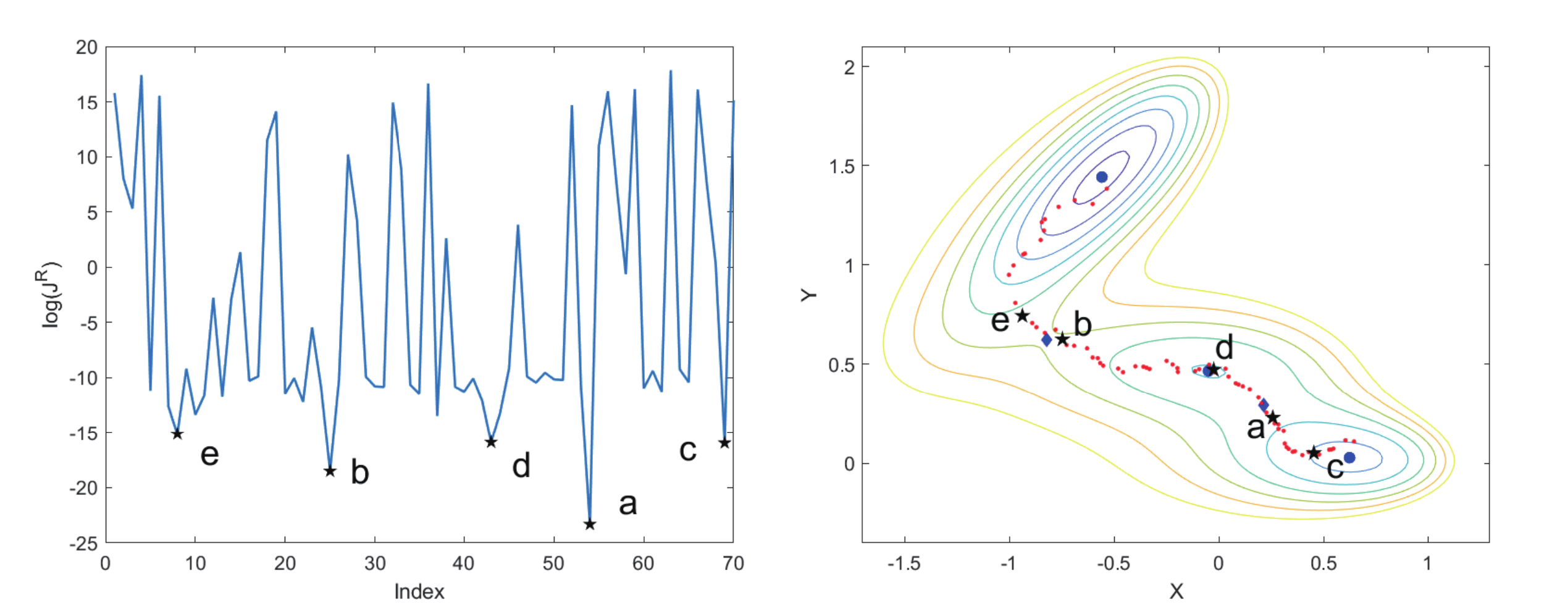}
  \caption{The key transition states along dominant transition path, $\eps =0.1$. Left panel: the $\log J^R$ along dominant transition path. The states with the least five $J^R$ are marked by a-e in ascending order. Right panel: the projection of dominant transition path on $\mathbb{R}^2$. This sub-figure is similar as (d)-(f) of Fig. \ref{fig:MuellerSphere} except that the dominant transition path is marked with red dots. The five states in the left panel are also marked in this sub-figure.
   }\label{fig:BN}
\end{figure}

The transition rate calculated by \eqref{eq:kab} is also consistent with the Freidlin-Wentzell theory. When $A$, $B$ are metastable states, it is well known that as $\eps \to 0$, $\eps \log k_{AB} \to S(B;A)$, where $S(B;A)$ in \eqref{eq:quasipotential} is the so-called quasi-potential. The value of $S(B;A)$ is a side product when computing the minimum energy path on manifold $\nn$ by MAM. The rescaled logarithm of the rates $k_{AB}$ for different $\eps $ and $S(B;A)$  are listed in Table~\ref{tab:1}. We can find  that as $\eps $ becomes smaller, these two quantities  get closer, as suggested by the theoretical result.

\begin{table}[htbp]
  \centering
  \begin{tabular}{c|c|c|c|c|c}
    \hline\hline
     & $\eps = 1$ & $\eps = 0.2$ & $\eps = 0.05$ & $\eps =0.02$ &$S(B;A)$ \\
    \hline
    $\eps \log k_{AB}$ & -0.4282 & 0.2540 & 0.3979 & 0.3999 & 0.3816 \\
    \hline\hline
  \end{tabular}
  \vspace*{0.2cm}
  \caption{Comparison of the transition rates obtained by TPT and the quasi-potential.}\label{tab:1}
\end{table}

\paragraph{\bf Example 2: Mueller potential on torus.} We can also map the Mueller potential to the torus  $\mathcal{N}=\mathbb{T}^2$ which is defined as
\[x=(R+r\cos\theta)\cos\phi,\quad y=(R+r\cos\theta)\sin\phi, \quad z=r\sin\theta, \quad \theta\in [-\pi, \pi), \phi\in [0,2\pi).\]
Set $R=2.0, r=1.0$. We define the potential $U_\nn$ on torus as
\[U_{\nn}(x,y,z) =U_{\mathbb{T}^2}(x,y,z) := U(r\theta, R\phi).\]

It is interesting to study the transition behavior of dynamics \eqref{sde-y} with finite but small noise and the driving potential are also  perturbed by noise with similar scales. In this case, minimizing the Freidlin-Wentzell action functional is not proper because the effect of finite noise is ignored. Finite temperature string method \cite{FTSM} is a good candidate for this problem. However,  it is not straightforward to apply this method on a manifold.

Instead, we can still study this problem by TPT on point clouds. We perturb the Mueller potential by small oscillations as
\[\tilde{U}(X,Y) = U(X,Y) + 0.15\sin(10\pi X)\sin(10\pi Y)\]
and define the perturbed potential $U_\nn$ on torus by
\[U_{\nn}(x,y,z) = \tilde{U}_{\mathbb{T}^2}(x,y,z) := \tilde{U}(r\theta, R\phi).\]
We choose $\eps =0.1$, which is in the same scale of our perturbations, and consider the transitions from $A=D\cap \mathcal{B}^{\mathbb{R}^3}_{0.05}(\vec{X}_1)$ to $B=D\cap \mathcal{B}^{\mathbb{R}^3}_{0.05}(\vec{X}_3)$. By using 4000 uniform random samples on $\mathbb{T}^2$, we obtain the dominant transition paths  as shown with red circles in Fig.~\ref{fig:MuellerTorusRough}. For a reference, we still plot the minimum energy path under zero noise and zero perturbation in (b).

\begin{figure}[htbp]
  \centering
  \includegraphics[angle = 90, width=1\textwidth]{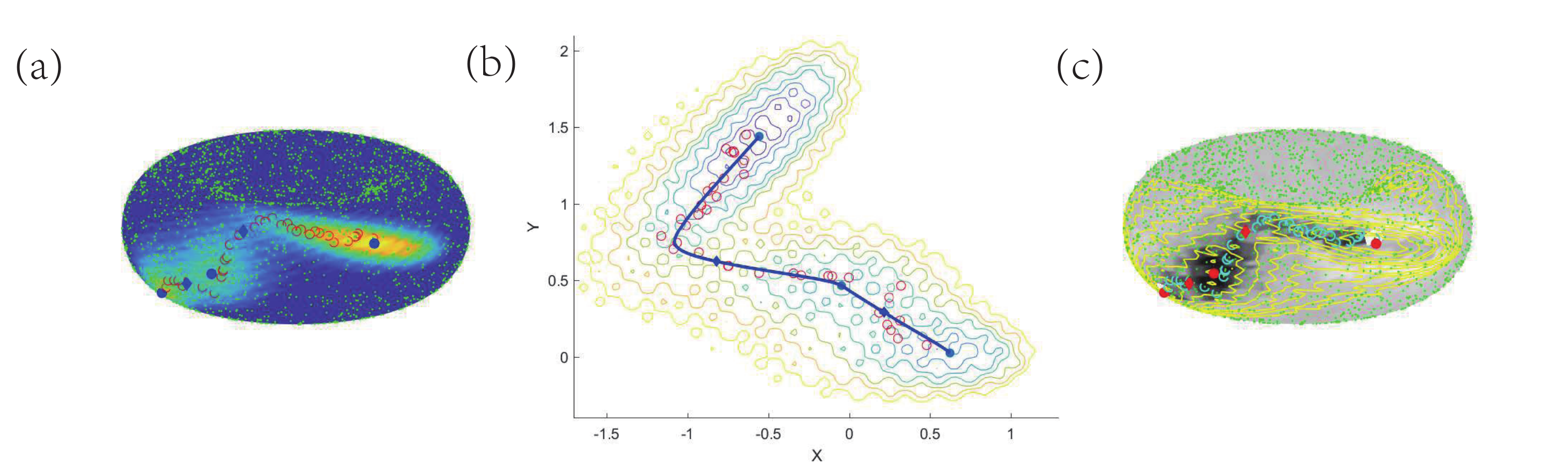}
  \caption{The dominant transition path from $\vec{X}_1$ to $\vec{X}_3$ on torus under perturbed Mueller potential ($\eps =0.1$). In three sub-figures, The metastable states and saddle points are shown with blue dots and diamonds, respectively. (a) The background shows the heat plot of $\tilde{U}_{\mathbb{T}^2}(x,y,z)$.  (c) The background shows the density of transition paths $\rho^R$ in log scale. The yellow lines correspond to the contour  of $\tilde{U}_{\mathbb{T}^2}(x,y,z)$. In both sub-figures, the small green dots are 4000 random samples, and the path of red/cyan circles is the dominant transition path. (b) The background is the contour plot of the perturbed Mueller potential $\tilde{U}(X,Y)$. The red circles and  blue solid line correspond to the projection of dominant transition path and the reference zero-noise minimum energy path, respectively.}\label{fig:MuellerTorusRough}
\end{figure}

We  can find that the noise effect on the potential is eliminated, and we still capture the main transition behavior from $\vec{X}_1$ to $\vec{X}_3$. Although the landscape is rough, the density of transition path $\rho^R$ is relatively smooth. The dominant transition path lies in the domain with largest value of $\rho^R$.

\subsection{Application on an alanine dipeptide}
We now   apply our method to a computational chemistry problem, a   manageable alanine dipeptide example with 22 atoms,  with collected data from molecular dynamics (MD) simulation.
\paragraph{\bf Example 3: Application on alanine dipeptide in vacuum.} The alanine dipeptide in vacuum is a simple and well studied molecule with  22 atoms, which implies $p=66$. It has been shown that the lower energy states of alanine dipeptide can be mainly described by two backbone dihedral angles $\phi\in [-\pi, \pi)$ and $\psi\in [-\pi, \pi)$ (as shown in Fig. \ref{fig:aa}); see \cite{Apostolakis1999Calculation}. Thus, its configuration essentially lies on a torus $\mathcal{N}$, and the dynamics  can be approximately governed by a stochastic equation like \eqref{sde-y} with $\ell=3$.  The transition between different isomers of alanine dipeptide is a good example for the study of rare events.

\begin{figure}
  \centering
  \includegraphics[width=5cm]{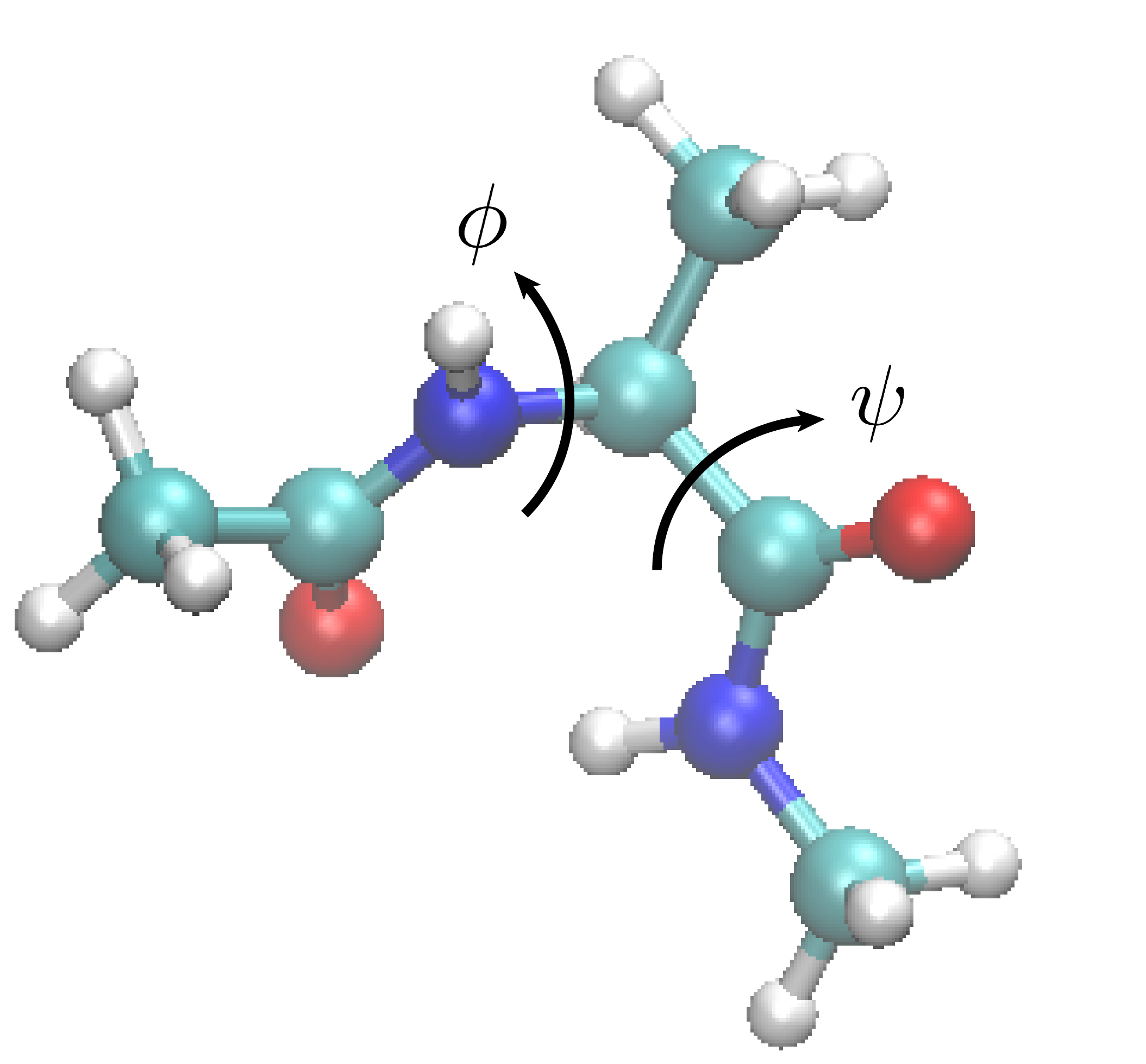}
  \caption{Schematic representation of the alanine dipeptide and two backbone dihedral angles $\phi$ and $\psi$.
  }\label{fig:aa}
\end{figure}

We apply the full atomic MD simulation of the alanine dipeptide molecule in vacuum with the AMBER99SB-ILDN force field for 100ns with room temperature $T=298K$. Then collect the data $(\phi,\psi)$ directly from the MD result. The free energy $U_{\mathcal{N}}(\phi,\psi)$ are obtained from the MD simulation by the reinforced dynamics (Fig.~\ref{fig:PE}), which approximates $U_{\mathcal{N}}$ through a deep neural network and utilizes the adaptive biasing force and reinforcement learning idea to encourage the space exploration \cite{ZWE}. There are three local minima $C_{ax}$, $C_{7eq}$ and $C'_{7eq}$, corresponding to different isomers of alanine dipeptide. We will study the dominant transition paths from  $C_{ax}$ to $C_{7eq}$, and $C_{ax}$ to $C'_{7eq}$ in the following with the developed algorithms.

\begin{figure}
  \centering
  \includegraphics[angle=90, width=0.7\textwidth]{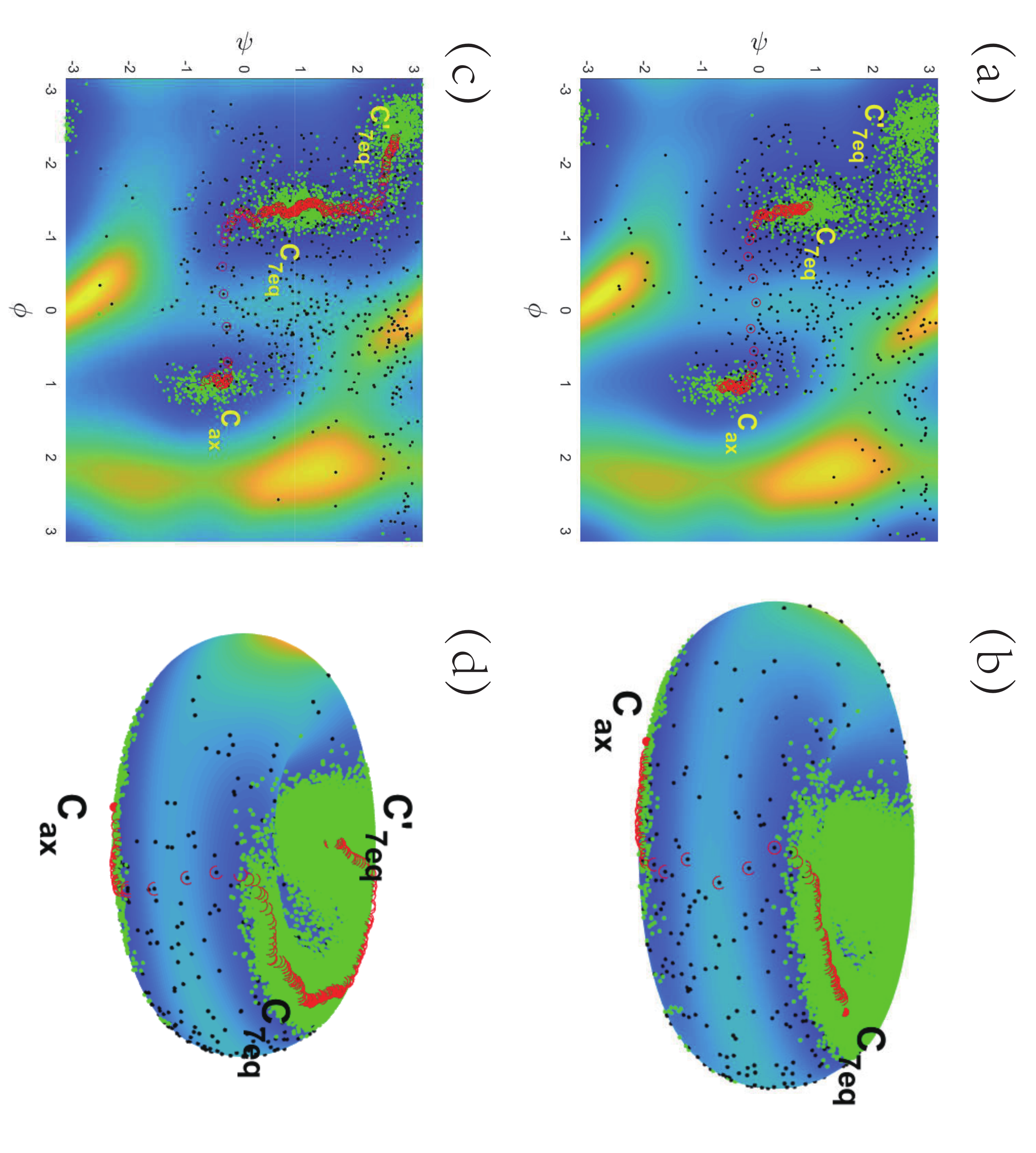}
  \caption{The sample points and transition paths from $C_{ax}$ to $C_{7eq}$ (subplot (a)-(b)) and  $C_{ax}$ to $C'_{7eq}$ (subplot (c)-(d)).  The green and black dots are MD and generated auxiliary sample points, respectively. The transition paths are marked by red circles. The background shows the effective potential $U_{\mathcal{N}}$ obtained by MD simulation. (a),(c): visualization in the $(\phi,\psi)$ plane. (b),(d): visualization on the torus.}\label{fig:PE}
\end{figure}

We collect $n=50,000$ equidistant MD time series data $D=\{(\phi_{t_i}, \psi_{t_i})\}_{i=1:n}$ with $0=t_0<t_1<\cdots <t_{n}$ ($t_{i+1}-t_i=\text{Const.}$) and map them to the torus $\mathcal{N}=\mathbb{T}^2\subset\mathbb{R}^3$ by $\sigma: (\phi, \psi)\to (x,y,z)$ defined as
\[x_{t_i}=(R+r\cos\phi_{t_i})\cos\psi_{t_i},\quad y_{t_i}=(R+r\cos\phi_{t_i})\sin\psi_{t_i}, \quad z_{t_i}=r\sin\phi_{t_i},\]
where $R=2, r=1$. These data are shown in Fig.~\ref{fig:PE} (green dots). One can find that they concentrate around the metastable states and seldom appear in the transition region ($z\approx 0$ in Fig.~\ref{fig:PE}(b),(d)). To study the transition path, we need to do data enrichment by generating some auxiliary data.

We do it by interpolation in the following way. Firstly, find the indices $I=\{1\leq j\leq n\big|z_{t_j}>0, z_{t_{j+1}}<0\}$ and collect two data sets $D^+ = \{(\phi_{t_i}, \psi_{t_i})|i=j-10,...,j, j\in I\}$ and $D^- = \{(\phi_{t_i}, \psi_{t_i})|i=j,..., j+10, j\in I\}$. Then we randomly select a pair of samples $(\phi^+_{i}, \psi^+_{i})\in D^+$ and $(\phi^-_{i}, \psi^-_{i})\in D^-$. An auxiliary sample $(\tilde{\phi}_{i}, \tilde{\psi}_{i})$ is generated by $\tilde{\phi}_{i}=\beta_1\phi^+_{i}+(1-\beta_1)\phi^-_{i}$; $\tilde{\psi}_{i}=\beta_2\psi^+_{i}+(1-\beta_2)\psi^-_{i}$, where $\beta_{1,2}\sim \mathcal{U}[0,1]$.

We sparsify the MD data points by randomly choosing 4,000 samples in $D$ and generating 500 auxiliary samples $\{(\tilde{\phi}_i, \tilde{\psi}_i)\}_{i=1:500}$. Namely, the dataset we use to compute the dominant transition path is $\tilde{D}=\{\sigma(\phi_{t_i},\psi_{t_i})\}_{i\in B}\cup\{\sigma(\tilde{\phi}_i, \tilde{\psi}_i)\}_{i=1:500}\subset\mathcal{N}$, where $B\subset \{1,2,\dots,n\}$ is a random batch with size 4000. The auxiliary samples are shown in Fig.~\ref{fig:PE} with black dots. Applying TPT theory, we get the dominant transition path from $C_{ax}$ to $C_{7eq}$, shown in    Fig.~\ref{fig:PE}(a)-(b) with red circles. Similar approach can also be applied to obtain the dominant transition path from $C_{ax}$ to $C'_{7eq}$ (Fig.~\ref{fig:PE}(c)-(d)). Both results are consistent with previous studies on this problem with other methods \cite{Apostolakis1999Calculation, Ren2005Transition}.

\begin{figure}
  \centering
  \includegraphics[angle=90, width=\textwidth]{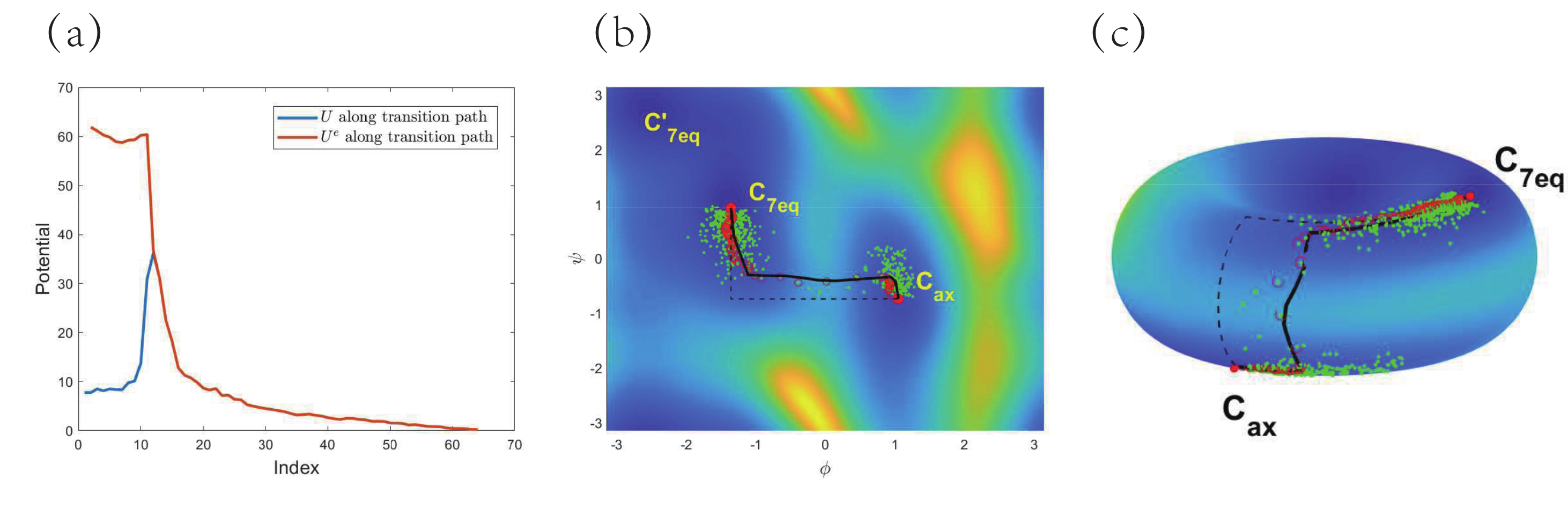}
  \caption{Numerical results for the transition from $C_{ax}$ to $C_{7eq}$ under the controlled process. (a) Comparison of  $U_{\mathcal{N}}$ and effective potential $U^e$ along the dominant transition path. (b)-(c) Simulation data of the controlled random walk \eqref{mp_oc} and the mean transition path obtained by Algorithm \ref{alg:meanpath} with different visualizations (in the $(\phi,\psi)$ plane or on the torus). The green dots are Monte Carlo samples. The dominant transition path is shown in red circles. The black solid and dashed lines correspond to the mean transition path computed by Algorithm \ref{alg:meanpath} and the initial path, respectively. }\label{fig:ctrled}
\end{figure}

With the help of controlled random walk \eqref{mp_oc}, we can simulate the transitions between the isomers more efficiently. We use the same dataset $\tilde{D}$ as in previous computation and perform the simulation for studying the transition from $C_{ax}$ to $C_{7eq}$ by Monte Carlo Algorithm \ref{alg:RW}. The transition happens 21 times in $K=10^5$ simulation steps. The potential $U_{\mathcal{N}}$ and effective potential $U^e$ along the dominant transition path in Fig.~\ref{fig:PE}(a)-(b) are shown in Fig.~\ref{fig:ctrled}(a). Similar as the case in Fig.~\ref{fig:control}, the effective potential $U^e$ achieves the local maximum at $C_{ax}$ state, and approximately decreases to $C_{7eq}$ state. In contrast, the original potential $U_{\mathcal{N}}$ has a sharp local minimum at $C_{ax}$ state, which results in the rare transition from $C_{ax}$ to $C_{7eq}$. This difference makes the transition under the potential $U^e$ is easy and frequent.

We then apply the Algorithm \ref{alg:meanpath} to get the mean transition path via the Monte Carlo samples of controlled random walk. We set $M=100$ and $L_{\rm max}=200$ in Algorithm \ref{alg:meanpath}. The numerical results are shown in Fig.~\ref{fig:ctrled}(b)-(c). One can find that the mean transition path is perfectly consistent with the dominant transition path obtained by TPT.

\section{Conclusion}
 In this paper, we first reinterpreted the transition state theory and the transition path theory as optimal control problems in an infinite time horizon. At a finite noise level $\eps>0$, based on the associated optimal control $v^*=2\eps \nabla \ln q$ and the controlled  effective equilibrium $\pi^e = q^2 \pi$, we design an optimally controlled random walk on point clouds, which realizes the original rare events almost surely  in $O(1)$ time scale. This enables an efficient sampling for the transitions between two conformational  states in a biochemical reaction system.   Taking advantage of the level set of the committor function $q$ and  the effective equilibrium $\pi^e = q^2 \pi$, a local averaging algorithm  is proposed to compute the mean transition path on manifold efficiently via the controlled Monte Carlo simulation data. Both synthetic and real world examples are conducted to show the efficiency of the  proposed algorithms, which gives consistent results with the dominant transition path algorithm.   Rigorously showing this consistency in mathematical aspect is an interesting problem in the future.

\section*{Acknowledgements}
The authors would like to thank  Prof. Weinan E for valuable suggestions, and Yuzhi Zhang for the help on the MD simulation of alanine dipeptide. J.-G. Liu was supported in part by NSF under award DMS-2106988. T. Li  was supported by the NSFC under grants Nos. 11421101 and 11825102, and Beijing Academy of Artificial Intelligence (BAAI). X. Li was supported by the construct program of the key discipline in Hunan Province.

\bibliographystyle{alpha}
\bibliography{tst}
\end{document}